\documentclass[10pt]{amsart}
\usepackage{amsmath,amssymb, amsthm, amsbsy}
\usepackage[dvips]{graphicx}
\usepackage[usenames,dvipsnames]{pstricks}
\usepackage{epsfig}
\usepackage{pst-grad} 
\usepackage{pst-plot} 

\newcommand{\R}{{\mathbb R}}

\newcommand{\C}{{\mathbb C}}

\renewcommand{\geq }{\geqslant}

\renewcommand{\leq }{\leqslant}

\numberwithin{equation}{section}
\newtheorem{theorem}{Theorem}[section]
\newtheorem{corollary}[theorem]{Corollary}
\newtheorem{lemma}[theorem]{Lemma}
\newtheorem{proposition}[theorem]{Proposition}
\theoremstyle{definition} \newtheorem{definition}[theorem]{Definition}
\newtheorem{example}[theorem]{Example}

\newtheorem{remark}[theorem]{Remark}

\begin{document}
\title[Hardy uncertainty principle in magnetic fields]{Hardy uncertainty principle and unique continuation properties of covariant Schr\"odinger flows}


\author{J.A. Barcel\'o}
\address{Juan Antonio Barcel\'o: ETSI de Caminos, Universidad Polit\'ecnica de Madrid, 28040 Madrid, Spain
}
\email{juanantonio.barcelo@upm.es}

\author{L. Fanelli}
\address{Luca Fanelli: SAPIENZA Universit$\grave{\text{a}}$ di Roma, Dipartimento di Matematica, P.le A. Moro 5, 00185-Roma, Italy}
\email{fanelli@mat.uniroma1.it}

\author{S. Guti\'errez}
\address{Susana Guti\'errez: School of Mathematics, The Watson Building, University of Birmingham, Edgbaston, Birmingham, B15 2TT, England}
\email{S.Gutierrez@bham.ac.uk}

\author{A. Ruiz}
\address{Alberto Ruiz: Departamento de Matem\'aticas, Universidad Aut\'onoma de Madrid, 28049,
Madrid, Spain}
\email{alberto.ruiz@uam.es}

\author{M.C. Vilela}
\address{Mari Cruz Vilela: Universidad de Valladolid, Departamento de Matem\'atica Aplicada, E.U. de Inform\'atica,  Plza. de Santa Eulalia 9 y 11, 40005 Segovia, Spain.}
\email{maricruz@eii.uva.es}

\subjclass[2000]{35J10, 35L05.}
\keywords{Schr\"odinger equation, electromagnetic potentials, Carleman estimates, uncertainty principle}

\thanks{
The first and fourth authors were supported by the Spanish projects MTM2008-02568 and MTM2011-28198,
the third by the Spanish projects MTM2007-62186 and MTM2011- 24054,
and finally, the fifth by Spanish projects MTM2007-62186 and MTM2011-28198.}

\begin{abstract}
We prove a logarithmic convexity result for exponentially weighted $L^2$-norms of
solutions to electromagnetic Schr\"odinger equation, without needing to assume
smallness of the magnetic potential. As a consequence, we can prove a unique
continuation result in the style of the Hardy uncertainty principle, which
generalize the analogous theorems which have been recently proved by Escauriaza, Kenig, Ponce and Vega.
\end{abstract}

\date{\today}
\maketitle

\section{Introduction}\label{sec:introduction}

This paper is concerned with sharp decay profiles, at two distinct times, of
$L^2$-solutions to an electromagnetic Schr\"odinger equation of the type
\begin{equation}\label{eq:main}
\partial_tu=i\left(\Delta_A +V\right)u,
\end{equation}
where $u=u(x,t):\R^{n+1}\to\C$, $A=A(x):\R^n\to\R^n$, $V=V(x,t):\R^n\to\C$, and  we use the notations
\begin{equation*}
\nabla_A:=\nabla-iA,
\qquad
\Delta_A:=(\nabla-iA)^2.
\end{equation*}
Our main goal is to start with a project devoted to understand
sufficient conditions on solutions to \eqref{eq:main}, the
coefficients $A,V$, and the behavior of the solutions at two
different times which ensure the rigidity $u\equiv0$. This follows a
program which has been developed for the magnetic free case $A\equiv
0$ by Escauriaza, Kenig, Ponce and Vega in the last few years in the
sequel of papers \cite{EKPV0, EKPV1, EKPV2, EKPV3, EKPV4}, and more
recently with Cowling in \cite{CEKPV}. Their main motivations is the
connection between unique continuation properties of Schr\"odinger
evolutions and the so called {\it Hardy uncertainty principle} (see
e.g. \cite{SS}), which can be briefly stated as follows:

{\it If $f(x)=O\left(e^{-|x|^2/\beta^2}\right)$ and its Fourier transform
$\hat f(\xi)=O\left(e^{-4|\xi|^2/\alpha^2}\right)$, then}
\begin{align*}
  \alpha\beta<4
  &
  \Rightarrow
  f\equiv0
  \\
  \alpha\beta=4
  &
  \Rightarrow
  f\ is\ a \ constant\ multiple\ of\ e^{-\frac{|x|^2}{\beta^2}}.
\end{align*}
Due to the strict connection between the Fourier transform $\mathcal F$ and the solution to the free Schr\"odinger equation with initial datum $f$ in $L^2$, namely
\begin{equation*}
  u(x,t):=e^{it\Delta}f(x) = (2\pi it)^{-\frac{n}{2}}e^{i\frac{|x|^2}{4t}}
  \mathcal F\left(e^{i\frac{|\cdot|^2}{4t}}f\right)\left(\frac{x}{2t}\right),
\end{equation*}
 the Hardy uncertainty principle has a PDE's-counterpart, which can be stated as follows:

{\it If $u(x,0)=O\left(e^{-|x|^2/\beta^2}\right)$ and $u(x,T):=e^{iT\Delta}u(x,0)=
O\left(e^{-|x|^2/\alpha^2}\right)$, then}
\begin{align*}
  \alpha\beta<4T
  &
  \Rightarrow
  u\equiv0
  \\
  \alpha\beta=4T
  &
  \Rightarrow
  u(x,0)\ is\ a \ constant\ multiple\ of\ e^{-\left(\frac{1}{\beta^2}+\frac{i}{4T}\right)|x|^2}.
\end{align*}
The corresponding $L^2$-versions of the previous results were proved in \cite{SST} and affirm the following:
\begin{align*}
  e^{|x|^2/\beta^2}f \in L^2,\  e^{4|\xi|^2/\alpha^2}\hat f \in L^2,\
  \alpha\beta\leq4
  &
  \Rightarrow
  f\equiv0
  \\
  e^{|x|^2/\beta^2}u(x,0)\in L^2,\  e^{|x|^2/\alpha^2}e^{iT\Delta}u(x,0) \in L^2,\
  \alpha\beta\leq4T
  &
  \Rightarrow
  u\equiv0.
\end{align*}
Obviously, without loss of generality, we might restrict our attention to the case $T=1$.
An interesting survey about this topic can be found in \cite{BD}.

One of the major contributions of the authors of \cite{CEKPV, EKPV0,
EKPV1, EKPV2, EKPV3, EKPV4} was to deeply understand the relation
between these kind of properties and the logarithmic convexity
property of exponentially weighted $L^2$-norms of solutions to
Schr\"odinger equations (see also \cite{EFV}, \cite{EKPV5} for
analogous results concerning unique continuation from the infinity).
This permits to perform purely real analytical proofs, and then
allows rough coefficients in the differential equations, which are
difficult to be handled by Fourier transform or general complex
analysis tools. For example, in \cite{EKPV2} and \cite{EKPV3} the
authors considered any bounded potential of the form
$V=V_1(x)+V_2(x,t)$, possibly being $V_2$ complex-valued, without
assuming any Sobolev regularity and any smallness condition on the
two components; in this situation, they were able to establish the
analog to the above statements, in the cases $\alpha\beta<2$ first
(\cite{EKPV2}), and the sharp $\alpha\beta<4$ later (\cite{EKPV3}),
for $T=1$. The strategy can be roughly summarized as follows:
\begin{itemize}
\item
Assume $e^{|x|^2/\beta^2}u(x,0),\
e^{|x|^2/\alpha^2}e^{i\Delta}u(x,1) \in L^2$ and prove a
logarithmic convexity estimate for the quantity
$H(t):=\|e^{g(x,t)}u(t)\|_{L^2}$ of the type $H(t)\lesssim
H(0)^tH(1)^{1-t}$, where $g$ is a suitable function, bounded
with respect to $t$ and quadratically growing with respect to
$x$. This shows that a gaussian decay at times 0 and 1 is
preserved (and in fact improved) for intermediate times.
\item
Start a self-improvement argument, by suitably moving the center of the gaussian as $e^{a(t)|x+b(t)|^2}$, based on analytical estimates (Carleman estimates; this leads up to the non-sharp result $\alpha\beta<2$, see \cite{EKPV2}) or on the logarithmic convexity itself (this leads to the sharp result $\alpha\beta<4$, see \cite{EKPV3}), which finally gives the rigidity $u\equiv0$.
\end{itemize}
Amazingly, proving these results in a rigorous way represents not just a considerable technical difficulty, but also a conceptual obstacle which, if avoided, could bring to misleading results. To overcome this problem, the above mentioned authors introduce a small artificial dissipation term in the equation which turns out to be fundamental, and finally let it tend to zero.

It is quite natural to claim that, with some efforts, some small
first-order terms can be introduced in the argument by Escauriaza,
Kenig, Ponce and Vega without loosing the results. The aim of this
paper is to understand in which way the uncertainty can be described
in the presence of first-order perturbations in covariant form, as
in \eqref{eq:main} when $A\neq0$. Precisely, our goal here is to
obtain similar results
without assuming any smallness conditions on $A$ and possibly
respecting the mathematical properties of the quantities which are
behind these kinds of models.

We continue to introduce the terminology and notation required to
state the main results of this paper.

Let $A=(A^1(x),\dots, A^n(x)):\R^n\to\R^n$, with $n\geq2$, be a
vector field, which we will usually refer to as a magnetic
potential, and denote the magnetic field by $B(x)=DA(x)-DA^t(x)$,
the anti-symmetric gradient of $A$, namely
\begin{equation*}
  B\in\mathcal M_{n\times n}(\R),
  \qquad
  B_{jk}(x):=A^k_j(x)-A^j_k(x).
\end{equation*}
From now on, given a scalar function $f$, we always use the notation $f_k(x)=\partial_{x_k}f(x)$, while an upper index will denote the component of a vector.
In dimension $n=2$, $B$ is identified with the scalar function $B\widetilde =\text{curl}\,A=A^1_2-A^2_1$; the same identification holds in dimension $n=3$, where now $\text{curl}\,A$ is a vector field and
\begin{equation*}
v^tB = v\times\text{curl}\,A
\qquad
\forall v\in\R^3,
\end{equation*}
the cross denoting the vectorial product in $\R^3$. Since equation
\eqref{eq:main} is gauge invariant (see Section
\ref{subsec:cronstrom} below), it is always important to keep in
mind that the physically meaningful quantity is the magnetic field
$B$, while the potential $A$ is a mathematical construction. This
fact has to be considered when we state a theorem, since a
meaningful result should not depend on a particular choice of the
gauge.

As it will be clear in the sequel, another relevant object is the
vector-field $\Psi(x):=x^tB(x)$; in 3D, it can be interpreted,
modulo its intensity, as a tangential projection of
$\text{curl}\,A$, since
\begin{equation*}
  x^tB(x)=x\times \text{curl}\,A(x) = |x|B_\tau(x),
  \qquad
  n=3
\end{equation*}
following the notation $B_\tau=\frac{x^t}{|x|}B$ introduced in \cite{FV}. As we see in \eqref{eq:assA3hardy} below, $x^tB$ is essentially the only component of $B$ on which one needs to make suitable assumptions, in order to obtain a Hardy uncertainty principle.

We can now state the main result of this paper.

\begin{theorem}\label{thm:hardy}
  Let $n\geq 3$, and let $u\in\mathcal C([0,1];L^2(\R^n))$
  be a solution to
  \begin{equation}\label{eq:main3hardy}
    \partial_tu
    =
    i\left(\Delta_A+V_1(x)+V_2(x,t)\right)u
  \end{equation}
  in $\R^n\times[0,1]$, with $A=A(x):\R^n\to\R^n$, $V_1=V_1(x):\R^n\to\R$, $V_2=V_2(x,t):\R^{n+1}\to\C$. Assume that
  \begin{equation}\label{eq:cronstrom13hardy}
    \int_0^1A(sx)\,ds\in\R^n
  \end{equation}
  is well defined at almost every $x\in\R^n$.
Moreover, denote by $B=B(x)=DA-DA^t$, $B_{jk}=A^k_j-A^j_k$ and assume that there exists a unit vector $v\in\mathcal S^{n-1}$ such that
  \begin{equation}\label{eq:ortoghardy}
  v^tB(x)\equiv0.
  \end{equation}¥
  In addition, assume that
  \begin{align}
    \label{eq:assA3hardy}
    &
    \|x^tB\|_{L^\infty}^2
    :=
    \frac{M_A}{4}<\infty
    \\
    \label{eq:assV3hardy}
    &
    \|V_1\|_{L^\infty}
    :=M_1<\infty
    \\
    \label{eq:assV23hardy}
    &
    \sup_{t\in[0,1]}\left\|
    e^{\frac{|\cdot|^2}{(\alpha t+\beta
    (1-t))^2}}V_2(\cdot,t)\right\|_{L^\infty}
    e^{\sup_{t\in[0,1]}\left\|\Im V_2(\cdot,t)\right\|
    _{L^\infty}}
    :=M_2<\infty
    \\
    \label{eq:decay3hardy}
    &
    \left\|e^{\frac{|\cdot|^2}{\beta^2}}
    u(\cdot,0)\right\|_{L^2}+
    \left\|e^{\frac{|\cdot|^2}{\alpha^2}}
    u(\cdot,1)\right\|_{L^2}<\infty,
  \end{align}
  for some $\alpha,\beta>0$. Then, $\alpha\beta<2$ implies $u\equiv0$.
\end{theorem}
\begin{remark}\label{rem:self}
  Among various consequences, conditions \eqref{eq:cronstrom13hardy}, \eqref{eq:assA3hardy} and  \eqref{eq:assV3hardy}
 imply the self-adjointness in $L^2$ of the hamiltonian $H_A:=-\Delta_A+V_1$, with form-domain $H^1(\R^n)$, after a suitable reduction to the so called Cr\"onstrom (or transversal) gauge (see Section \ref{subsec:cronstrom} and Proposition \ref{prop:self} below). Hence the
  Schr\"odinger flow $e^{itH_A}$ is well defined for any $t\in\R$ by the Spectral Theorem, and unitary in $L^2$, so that given $u(x,0)\in L^2$ there exists a unique solution $u\in\mathcal C([0,1];L^2(\R^n))$ of the integral equation
  \begin{equation*}
    u(\cdot,t) =
    e^{itH_A}u(\cdot,0)+\int_0^t
    e^{i(t-s)H_A}V_2(\cdot,s)u(\cdot,s)\,ds,
  \end{equation*}
  provided \eqref{eq:assV23hardy}.

  In addition, also the heat flow $e^{tH_A}$ is well defined for positive times, and this will be used in the sequel.
\end{remark}
\begin{remark}\label{rem:kernel}
Notice that no smallness conditions on $A,V_1,V_2$ are required in the statement of Theorem
\ref{thm:hardy}. On the other hand,
  condition \eqref{eq:ortoghardy} naturally comes into play once we prove a Carleman estimate (Lemma \ref{thm:carleman} below), which is one of the tools to prove Theorem \ref{thm:hardy}. We remark that we cannot prove the result in dimension $n=2$, which remains as an open question, since there are no $2\times2$-antisymmetric matrices with non-trivial kernel.

The clearest examples of fields $B$ satisfying our assumptions can be constructed as follows. Denote by
  \begin{equation*}
    M_{2k-1}=
    \left(
    \begin{array}{cccc}
      J & \ & \ & \
      \\
      \  & \ddots & \ & \
      \\
      \ & \ & J & \
      \\
      \ & \ & 0 & 0
    \end{array}
    \right),
    \qquad
    M_{2k}=
    \left(
    \begin{array}{cccc}
      J & \ & \ & \
      \\
      \  & \ddots & \ & \
      \\
      \ & \ & J & \
      \\
      \ & \ & 0 & 0
      \\
      \ & \ & 0 & 0
    \end{array}
    \right),
  \end{equation*}
  with $J:=
    \left(
    \begin{array}{cc}
      0 & 1
      \\
      -1 & 0
    \end{array}
    \right)$ and $k\geq2$. Now define for $n\geq3$
    \begin{equation*}
      B(x) = \frac{x^t}{|x|^2}M_n
    \end{equation*}
  and notice that the assumptions of Theorem \ref{thm:hardy} are satisfied. In particular, in dimension $n=3$ we can identify the above $B$ with
  \begin{equation*}
    B(x) = |x|^{-2}(-x_2,x_1,0) = \text{curl}\left(\frac{x_1x_3}{|x|^2},\frac{x_2x_3}{|x|^2},
    \frac{-x_1^2-x_2^2}{|x|^2}\right).
  \end{equation*}
  This is, as far as we understand, also a quite interesting hint about the fact that Theorem
    \ref{thm:hardy} is presumably not true, in total generality, for any magnetic field $B$, and it will be matter of future work.
\end{remark}
\begin{remark}\label{rem:sharp}
  The constraint $\alpha\beta<2$ in Theorem \ref{thm:hardy} is far from the sharp $\alpha\beta<4$
  obtained in \cite{EKPV3} in the magnetic-free case $A\equiv0$. Actually, we use here the argument
  introduced in \cite{EKPV2}, involving the use of a Carleman estimate, which cannot lead to a better result.
  In addition, already at this level, we see the necessity of the condition \eqref{eq:ortoghardy}, which is a
  quite interesting fact. Presumably, when looking for the sharp result, some additional phenomena,
  involving the presence a non trivial magnetic field, should come into play. This will hopefully be
  matter of future work.
\end{remark}

Theorem \ref{thm:hardy} has several consequences regarding uniqueness of solutions to \eqref{eq:main}, both in the linear and nonlinear cases (i.e. $V_2(x,t)=|u(x,t)|^p$), which we will not investigate in this paper (see \cite{EKPV2, EKPV3} for details).

The main tool to prove theorem \ref{thm:hardy} is the following logarithmic convexity result.

\begin{theorem}[logarithmic convexity]\label{thm:3}
  Let $n\geq 2$, and consider a solution $u\in\mathcal C([0,1];L^2(\R^n))$
  to
  \begin{equation}\label{eq:main3}
    \partial_tu
    =
    i\left(\Delta_A+V_1(x)+V_2(x,t)\right)u
  \end{equation}
  in $\R^n\times[0,1]$, with $A=A(x):\R^n\to\R^n$, $V_1=V_1(x):\R^n\to\R$, $V_2=V_2(x,t):\R^{n+1}\to\C$.
  Denote by $B=B(x)=DA-DA^t$, $B_{jk}=A^k_j-A^j_k$ and assume that
  \begin{equation}\label{eq:cronstrom13}
    \int_0^1A(sx)\,ds\in\R^n
  \end{equation}
  is well defined at almost every $x\in\R^n$.
  Moreover, assume that
  \begin{align}
    \label{eq:assA3}
    &
    \|x^tB\|_{L^\infty}^2
    :=
    \frac{M_A}{4}<\infty
    \\
    \label{eq:assV3}
    &
    \|V_1\|_{L^\infty}
    :=M_1<\infty
    \\
    \label{eq:assV23}
    &
    \sup_{t\in[0,1]}\left\|
    e^{\frac{|\cdot|^2}{(\alpha t+\beta
    (1-t))^2}}V_2(\cdot,t)\right\|_{L^\infty}
    e^{\sup_{t\in[0,1]}\left\|\Im V_2(\cdot,t)\right\|
    _{L^\infty}}
    :=M_2<\infty
    \\
    \label{eq:decay3}
    &
    \left\|e^{\frac{|\cdot|^2}{\beta^2}}
    u(\cdot,0)\right\|_{L^2}+
    \left\|e^{\frac{|\cdot|^2}{\alpha^2}}
    u(\cdot,1)\right\|_{L^2}<\infty,
  \end{align}
  for some $\alpha,\beta>0$.
  Then, the function
  \begin{equation*}
    \Theta(t):=
    \log\left\|
    e^{\frac{|\cdot|^2}{(\alpha t+\beta(1-t))^2}}
    u(\cdot,t)\right\|_{L^2}^{\alpha t+\beta(1-t)}
  \end{equation*}
  is convex in $[0,1]$. In addition, there exists a constant $N=N(\alpha,\beta)>0$ such that
  \begin{align}
    \label{eq:31}
    &
    \left\|
    e^{\frac{|\cdot|^2}{(\alpha t+\beta(1-t))^2}}
    u(\cdot,t)\right\|_{L^2}
    \\
    \nonumber
    &
    \ \
    \leq
    e^{N\left[M_A+M_1+M_2+M_1^2+M_2^2\right]}
    \left\|e^{\frac{|\cdot|^2}{\beta^2}}
    u(\cdot,0)\right\|_{L^2}^{\frac{\beta(1-t)}
    {\alpha t+\beta(1-t)}}
    \left\|e^{\frac{|\cdot|^2}{\alpha^2}}
    u(\cdot,1)\right\|_{L^2}^{\frac{\alpha t}
    {\alpha t+\beta(1-t)}}
    \\
    \label{eq:32}
    &
    \left\|\sqrt{t(1-t)}
    e^{\frac{|x|^2}{(\alpha t+\beta(1-t))^2}}
    \nabla_Au(x,t)\right\|_{L^2(\R^n\times[0,1])}
    \\
    \nonumber
    &
    \ \
    \leq
    e^{N\left[M_A+M_1+M_2+M_1^2+M_2^2\right]}
    \left(
    \left\|e^{\frac{|\cdot|^2}{\beta^2}}
    u(\cdot,0)\right\|_{L^2}+
    \left\|e^{\frac{|\cdot|^2}{\alpha^2}}
    u(\cdot,1)\right\|_{L^2}
    \right).
  \end{align}
\end{theorem}
\begin{remark}
  Notice that in this case condition \eqref{eq:ortoghardy} is not needed; as a consequence, we can also handle the 2D case, which is included in the statement.
  We finally remark that both Theorems \ref{thm:hardy} and \ref{thm:3} hold in dimension $n=1$,  since in this case any reasonable magnetic potential can be gauged away by the Fundamental Theorem of Calculus.
\end{remark}

The   strategy of the proof of Theorem \ref{thm:3} is  as  follows:
\begin{itemize}
\item
by gauge transformation, we reduce to the case in which $x\cdot A\equiv0$ (see section
\ref{subsec:cronstrom}) below;
\item
we add a small dissipation term which regularizes the solution and gives a useful preservation property for the exponentially weighted $L^2$-norms of the solution (Lemma
\ref{lem:dissipation});
\item
by conformal (or Appell) transformation (see Lemma \ref{lem:appell}), we reduce to the case $\alpha=\beta$;
\item
we prove Theorem \ref{thm:3} in the case $\alpha=\beta$ (Lemmata \ref{lem:3}, \ref{lem:4});
\item
we translate the result in terms of the original solution, by inverting the conformal transformation, obtaining the final result.
\end{itemize}
Once Theorem \ref{thm:3} is proved, then Theorem \ref{thm:hardy} follows as an application of a Carleman inequality (Lemma \ref{thm:carleman}).

{\bf Acknowledgements.} The authors wish to thank Luis Escauriaza and Luis Vega for some useful discussions about the topic of the paper, and Vladimir Georgiev for addressing them to the topic of Section \ref{subsec:cronstrom} below.

\section{Preliminaries}\label{sec:preliminaries}

We devote this section to collect some preliminary results which will be needed in the proofs of our main results.

In order to prove the main theorems in a rigorous way, we need to add a dissipation term to equation
\eqref{eq:main}, which permits to assure that a gaussian decay at time 0 is preserved during the time evolution. For this reason, we study in this section some abstract properties regarding the solutions to
\begin{equation}\label{eq:1}
  \partial_tu=(a+ib)\left(\Delta_Au+V(x,t)u+F(x,t)\right),
\end{equation}
with $a,b\in\R$, $A=A(x,t):\R^{n+1}\to\R^n$, $V(x,t),F(x,t):\R^{n+1}\to\C$.

\subsection{The Cronstr\"om gauge}\label{subsec:cronstrom}
 Our first tool is the gauge invariance of equation \eqref{eq:1}.
We need to review some algebraic properties of magnetic Schr\"odinger operators, pointing our attention on the so called Cronstr\"om (or transversal) gauge.

Equation \eqref{eq:1} is {\it gauge invariant} in the following
sense: if $u$ solves \eqref{eq:1}, and we denote by $\widetilde
A=A+\nabla\varphi$, with $\varphi=\varphi(x):\R^n\to\R$, then the
function $\widetilde u=e^{i\varphi}u$ is a solution to
\begin{equation}\label{eq:2}
  \partial_t\widetilde u=(a+ib)\left(\Delta_{\widetilde A}\widetilde u+V(x,t)\widetilde u+e^{i\varphi}F(x,t)\right).
\end{equation}
Indeed, it is quite simple to verify that $\Delta_{\widetilde A}(e^{i\varphi}u)=e^{i\varphi}\Delta_Au$.

\begin{definition}\label{def:cronstrom}
  A connection $\nabla-iA(x)$ is said to be in the {\it Cronstr\"om gauge} (or {\it transversal gauge}) if $A\cdot x=0$, for any $x\in\R^n$.
\end{definition}

The following Lemma shows the transformation which permits to reduce a suitable potential to the Cronstr\"om gauge.

\begin{lemma}\label{lem:cronstrom1}
  Let $A=A(x)=(A^1(x),\dots,A^n(x)):\R^n\to\R^n$, for $n\geq2$ and denote by $B=DA-DA^t\in \mathcal M_{n\times
n}(\R)$, $B_{jk}=A^k_j-A^j_k$, and $\Psi(x):=x^tB(x)\in\R^n$. Assume
that the two vector quantities
  \begin{equation}\label{eq:cronstrom1}
    \int_0^1A(sx)\,ds\in\R^n,
    \qquad
    \int_0^1\Psi(sx)\,ds\in\R^n
  \end{equation}
  are finite, for almost every $x\in\R^n$; moreover, define the (scalar) function
  \begin{equation}\label{eq:varphi}
    \varphi(x):=x\cdot\int_0^1A(sx)\,ds\in\R.
  \end{equation}
  Then, the following two identities hold:
  \begin{align}\label{eq:cronstrom2}
    \widetilde A(x):=A(x)-\nabla\varphi(x)  & = -\int_0^1\Psi(sx)\,ds
    \\
    x^tD\widetilde A(x)  & = -\Psi(x) +\int_0^1\Psi(sx)\,ds.
    \label{eq:cronstrom3}
  \end{align}
\end{lemma}

\begin{proof}
  A simple proof of identity \eqref{eq:cronstrom2} can be found e.g. in \cite{I}. For the sake of completeness, we write it below.
  A direct computation shows that
  \begin{align*}
    \varphi_j(x)
    &
    = \frac{\partial}{\partial x_j}\varphi(x)
    = \int_0^1A^j(sx)\,ds
    +\int_0^1\sum_{k=1}^n
    sx_kA^k_j(sx)\,ds
    \\
    &
    =
    \int_0^1A^j(sx)\,ds
    +\int_0^1\sum_{k=1}^n
    sx_kA^j_k(sx)\,ds
    +\int_0^1\sum_{k=1}^n
    sx_kB_{jk}(sx)\,ds
    \\
    &
    =
    \int_0^1A^j(sx)\,ds
    +\int_0^1 s\frac{d}{ds}\left[
    A^j(sx)\right]\,ds
    +
    \int_0^1\Psi^j(sx)\,ds.
  \end{align*}
  Integrating by parts now yields \eqref{eq:cronstrom2}.

  We now pass to the proof of \eqref{eq:cronstrom3}. By \eqref{eq:cronstrom2}, we can now compute
  \begin{align*}
    & [D\widetilde A(x)]_{kj}=[D(A-\nabla\varphi)]_{kj}(x)
    \\
    & \ \
    =
    -\frac{\partial}{\partial x_k}
    \int_0^1\sum_{h=1}^nsx_h\left(A^j_h(sx)-A^h_j(sx)\right)\,ds
    \\
    & \ \
    = -\int_0^1s\left(A^j_k(sx)-A^k_j(sx)\right)\,ds
    -\sum_{h=1}^n
    \int_0^1s^2x_h\left(A^j_{h}-A^h_{j}\right)_k(sx)\,ds.
  \end{align*}
   Integrating by parts we obtain
  \begin{align*}
    &
    \left[x^tD\widetilde A(x)\right]^j
    =
    \sum_{k=1}^n
    x_k[D(A-\nabla\varphi)]_{kj}(x)
    \\
    &
    =
    -\sum_{k=1}^n\int_0^1sx_k
    \left(A^j_k(sx)-A^k_j(sx)\right)\,ds
    -\sum_{h=1}^n\sum_{k=1}^n
    \int_0^1s^2x_hx_k\left(A^j_{h}-A^h_{j}\right)_k(sx)\,ds
    \\
    &
    =-\sum_{k=1}^n\int_0^1sx_k
    \left(A^j_k(sx)-A^k_j(sx)\right)\,ds
    -\sum_{h=1}^n\int_0^1s^2x_h
    \frac{d}{ds}\left[A^j_h(sx)-A^h_j(sx)\right]\,ds
    \\
    &
    =\sum_{k=1}^n\int_0^1sx_k
    \left(A^j_k(sx)-A^k_j(sx)\right)\,ds
    -\sum_{k=1}^n
    x_k\left(A^j_k(x)-A^k_j(x)\right),
  \end{align*}
  which proves \eqref{eq:cronstrom3}.
\end{proof}
\begin{corollary}\label{cor:cronstrom}
  Under the same assumptions of Lemma \ref{lem:cronstrom1}, we have:
  \begin{equation}\label{eq:gaugefinal}
    x\cdot\widetilde A(x) \equiv0,
    \qquad
    x\cdot x^tD\widetilde A(x)\equiv0.
  \end{equation}
\end{corollary}
\begin{proof}
  The proof is a quite immediate consequence of \eqref{eq:cronstrom2}, \eqref{eq:cronstrom3}, and the fact that $B$ is an anti-symmetric matrix.
\end{proof}

\begin{remark}\label{rem:cronstromnew}
  Notice that conditions \eqref{eq:cronstrom13} and \eqref{eq:assA3} in Theorem \ref{thm:3} obviously imply \eqref{eq:cronstrom1}, hence Lemma \ref{lem:cronstrom1} and Corollary \ref{cor:cronstrom} are applicable under the assumptions of our main Theorems.
\end{remark}

\begin{example}[Aharonov-Bohm]\label{ex:AB}
  The following is possibly the most relevant example of a 2D-magnetic potential for which Lemma \ref{lem:cronstrom1} and Corollary \ref{cor:cronstrom} do not apply.
  Define the 2D-{\it Aharonov-Bohm} potential as
  \begin{equation*}
    A(x)=|x|^{-2}(-x_2,x_1).
  \end{equation*}
  In dimension $n=2$, the antisymmetric gradient $B=DA-DA^t$ is identified with the scalar quantity $B=\text{curl}\,A=A^1_2-A^2_1$.
  Writing
  \begin{equation*}
    A(x) = \nabla^{\bot}\log(|x|),
  \end{equation*}
  where $\nabla^\bot$ is the orthogonal gradient, chosen with the correct orientation,
  gives $B=\text{curl}\,A=\Delta\log(|x|)=2\pi\delta$.
  This shows that $\Psi(x) = x^tB(x)\equiv0$; if formula
  \eqref{eq:cronstrom2} were true in this case, it would give that $A\equiv0$, which is a contradiction. In fact, \eqref{eq:cronstrom1} does not hold in this case, since $A$ is too singular.

  In similar ways, it is possible to construct such examples of potentials $A$, in every dimension, satisfying $x^tB=0$ with $A\neq0$, which are not in contradiction with identity
  \eqref{eq:cronstrom2} since they do not satisfy  \eqref{eq:cronstrom1}.
  \end{example}
  
  \subsection{Self-adjointness}\label{subsec:self}
We now state a standard result about the self-adjointness of $H_A=-\Delta_A-V_1$.
\begin{proposition}\label{prop:self}
  Let $A=A(x)=(A^1(x),\dots,A^n(x)):\R^n\to\R^n$, $V_1=V_1(x):\R^n\to\R$ and denote by $ B=DA-DA^t\in \mathcal M_{n\times n}(\R)$, $B_{jk}=A^k_j-A^j_k$, and $\Psi(x):=x^tB(x)\in\R^n$, for $n\geq2$. Assume that  \begin{equation}\label{eq:cronstrom1self}
    \int_0^1A(sx)\,ds\in\R^n,
  \end{equation}
  is finite, for almost every $x\in\R^n$; moreover, assume that
  \begin{equation}\label{eq:thanks}
  V_1(x)\in L^\infty
  \qquad
  x^tB(x)\in L^\infty,
  \end{equation}
  and define $\widetilde A$ by \eqref{eq:cronstrom2}.
  Finally, consider the quadratic form
  \begin{align*}
    \widetilde q(\varphi,\psi)
    :=
    \int\nabla_{\widetilde A}\varphi\cdot\overline{\nabla_{\widetilde A}\psi}\,dx
    +\int V_1\varphi\overline\psi\,dx.
  \end{align*}
  Then $\widetilde q$ is the form associated to a unique self-adjoint operator
   $H_{\widetilde A}=-\Delta_{\widetilde A}-V_1(x)$, with form domain $H^1(\R^n)$.
\end{proposition}
\begin{proof}
The proof is completely standard. Indeed, notice that both $q$ and $\widetilde q$ are well defined on $H^1$, since $V_1\in L^\infty$ and $\widetilde A\in L^\infty$, thanks to
\eqref{eq:thanks} and Lemma \ref{lem:cronstrom1}. Moreover,  the norm
\begin{equation*}
  \||\psi|\|^2:=\widetilde q(\psi,\psi)+C\|\psi\|_{L^2}^2
\end{equation*}
is equivalent to the $H^1$-norm, for some $C>0$ sufficiently large, by the same reasons as above; this show that the form $\widetilde q$ is closed. Finally, the form is semibounded, i.e.
\begin{equation*}
  \widetilde q(\psi,\psi)\geq -C\|\psi\|_{L^2}^2,
\end{equation*}
by the same arguments. In conclusion, the thesis follows from Theorem VIII.15 in \cite{RS}.
\end{proof}

\subsection{The Appell transformation}\label{subsec:appell}

Following the strategy in \cite{EKPV2},
we now introduce a conformal transformation, usually referred to as the {\it Appell transformation}, as another tool for the proofs of our main results. As we see in the sequel, it permits to reduce matters in Theorem \ref{thm:3} to the situation in which $u(0)$ and $u(1)$ have the same gaussian decay, namely $\alpha=\beta$.
\begin{lemma}\label{lem:appell}
  Let $A=A(y,s)=(A^1(y,s),\dots,A^n(y,s)):\R^{n+1}\to\R^{n}$, $V=V(y,s),F=F(y,s):\R^n\to\C$, $u=u(y,s):\R^{n}\times[0,1]\to\C$ be a solution to
  \begin{equation}\label{eq:1appell}
  \partial_su=(a+ib)\left(\Delta_Au+V(y,s)u+F(y,s)\right),
\end{equation}
with $a+ib\neq0$,
and define, for any $\alpha,\beta>0$, the function
\begin{equation}\label{eq:appell}
  \widetilde u(x,t):=
  \left(\frac{\sqrt{\alpha\beta}}{\alpha(1-t)+\beta t}\right)^{\frac n2}
  u\left(\frac{x\sqrt{\alpha\beta}}{\alpha(1-t)+\beta t},\frac{t\beta}{\alpha(1-t)+\beta t}\right)
  e^{\frac{(\alpha-\beta)|x|^2}{4(a+ib)
  (\alpha(1-t)+\beta t)}}.
\end{equation}
Then $\widetilde u$ is a solution to
\begin{equation}\label{eq:2appell}
  \partial_t\widetilde u=(a+ib)\left(\Delta_{\widetilde A}\widetilde u
  +i\frac{(\alpha-\beta)\widetilde A\cdot x}{(a+ib)
  (\alpha(1-t)+\beta t)}\widetilde u+
  \widetilde V(x,t)\widetilde u+\widetilde F(x,t)\right),
\end{equation}
where
\begin{align}
  \label{eq:Aappell}
  \widetilde A(x,t)
  &
  = \frac{\sqrt{\alpha\beta}}{\alpha(1-t)+\beta t}
  A\left(\frac{x\sqrt{\alpha\beta}}{\alpha(1-t)+\beta t},\frac{t\beta}{\alpha(1-t)+\beta t}\right)
  \\
  \label{eq:Vappell}
  \widetilde V(x,t)
  &
  = \frac{\alpha\beta}{(\alpha(1-t)+\beta t)^2}
  V\left(\frac{x\sqrt{\alpha\beta}}{\alpha(1-t)+\beta t},\frac{t\beta}{\alpha(1-t)+\beta t}\right)
  \\
  \label{eq:Fappell}
  \widetilde F(x,t)
  &
  = \left(\frac{\sqrt{\alpha\beta}}{\alpha(1-t)+\beta t}\right)^{\frac n2+2}
  F\left(\frac{x\sqrt{\alpha\beta}}{\alpha(1-t)+\beta t},\frac{t\beta}{\alpha(1-t)+\beta t}\right)
  e^{\frac{(\alpha-\beta)|x|^2}{4(a+ib)
  (\alpha(1-t)+\beta t)}}.
\end{align}
\end{lemma}

\begin{proof}
  The proof is basically an explicit computation.
  Let us denote by
  \begin{align*}
    g(t)
    &
    := \frac{\sqrt{\alpha\beta}}{\alpha(1-t)+\beta t},
    \qquad
     c
    := \sqrt{\frac \beta\alpha},
    \qquad
    y=xg(t),
    \qquad
    s=ctg(t)
    \\
    h(x,t)
    &
    :=
    \frac{(\alpha-\beta)|x|^2}{4(a+ib)(\alpha(1-t)+\beta t)}
    =
    \frac{(\alpha-\beta)c}{4(a+ib)\beta}g(t)|x|^2.
  \end{align*}
  With these notations, we easily get
  \begin{align*}
    \widetilde u(x,t)
    &
    =
    g^{\frac n2}e^{h}u(y,s)
    \\
    \partial_t\widetilde u
    &
    =g^{\frac n2}e^{h}
    \left[g^2\partial_su+2(a+ib)g\nabla_xh\cdot\nabla_yu
    +\frac{(\alpha-\beta)c}{\beta}g\left(\frac n2+h\right)u\right].
  \end{align*}
  On the other hand, we have
  \begin{equation*}
    \nabla_x\widetilde u
    =
    g^{\frac n2}e^h\left(g\nabla_yu+u\nabla_xh\right),
  \end{equation*}
  and therefore
  \begin{equation*}
     g^{\frac n2}e^hg\nabla_yu
    =
    \nabla_x\widetilde u-\widetilde u\nabla_x h.
  \end{equation*}
  Moreover,
  \begin{align*}
    \Delta_x\widetilde u
    &
    =
    g^{\frac n2}e^h
    \left[g^2\Delta_yu+2g\nabla_xh\cdot\nabla_yu+
    \left(|\nabla_xh|^2+\Delta_xh\right)u\right]
    \\
    g^{\frac n2}e^hg^2\Delta_yu
    &
    =\Delta_x\widetilde u
    -2\nabla_xh\cdot\nabla_x\widetilde u
    +|\nabla_xh|^2\widetilde u
    -\widetilde u\Delta_xh.
  \end{align*}
  Now expand the operator $\Delta_A=(\nabla-iA)\cdot(\nabla-iA)$, in order to rewrite equation \eqref{eq:1appell} as
  \begin{equation}\label{eq:expand}
    \partial_su
    =
    (a+ib)\left(\Delta_yu-i(\text{div}_yA)u-2iA\cdot\nabla_yu-
    |A|^2u+V(y,s)u+F(y,s)\right);
  \end{equation}
  finally, since
  \begin{equation*}
    \widetilde A = gA,
    \qquad
    \text{div}_x\widetilde A=
    g^2\text{div}_yA,
    \qquad
    \widetilde V = g^2V,
    \qquad
    \widetilde F=g^{\frac n2+2}e^hF,
  \end{equation*}
  the thesis \eqref{eq:2appell} follows from \eqref{eq:expand} and the above identities.
\end{proof}

\begin{corollary}\label{cor:appell}
  With the same notations of Lemma \ref{lem:appell}, denoting by
  \begin{equation}\label{eq:sy}
     y=:\frac{\sqrt{\alpha\beta}x}{\alpha(1-t)+\beta t},
     \qquad
     s=:\frac{\beta t}{\alpha(1-t)+\beta t},
  \end{equation}
   we have, for any $\gamma\in\R$,
  \begin{align}
    \label{eq:corappell1}
    &
    \left\|e^{\gamma|\cdot|^2}\widetilde u(\cdot,t)\right\|
    _{L^2}
    =\left\|e^{\left[\frac{\gamma\alpha\beta}{(\alpha s
    +\beta(1-s))^2}+
    \frac{(\alpha-\beta)a}{4(a^2+b^2)(\alpha s+\beta(1-s))}\right]|\cdot|^2}u(\cdot,s)\right\|_{L^2}
    \\
    \label{eq:corappell2}
    &
    \left\|e^{\gamma|\cdot|^2}\widetilde F(\cdot,t)\right\|
    _{L^2}
    =
    \frac{\alpha\beta}{(\alpha(1-t)+\beta t)^2}
    \left\|e^{\left[\frac{\gamma\alpha\beta}{(\alpha s
    +\beta(1-s))^2}+
    \frac{(\alpha-\beta)a}{4(a^2+b^2)(\alpha s+\beta(1-s))}\right]|\cdot|^2}F(\cdot,s)\right\|_{L^2}
    \\
    \label{eq:corappell3}
    &
    \left\|\sqrt{t(1-t)}e^{\gamma|x|^2}
    \nabla_{\widetilde A}\widetilde u\right\|_{L^2(\R^n\times[0,1])}
    =
    \left\|\sqrt{s(1-s)}e^{\left[\frac{\gamma\alpha\beta}
    {(\alpha s
    +\beta(1-s))^2}+
    \frac{(\alpha-\beta)a}{4(a^2+b^2)(\alpha s+\beta(1-s))}\right]|y|^2}\right.
    \\
    \nonumber
    &
    \qquad\qquad\qquad\qquad\qquad \left.\times
    \left(\frac{\alpha s+\beta(1-s)}{\sqrt{\alpha\beta}}\nabla_Au
    +\frac{(\alpha-\beta)y}{2(a+ib)\sqrt{\alpha\beta}}u
    \right)\right\|_{L^2(\R^n\times[0,1])}
    \\
    \label{eq:corappell4}
    &
    \left\|\sqrt{t(1-t)}e^{\gamma|x|^2}
    |x|\widetilde u\right\|_{L^2(\R^n\times[0,1])}
    \\
    \nonumber
    & \ \ \
    =
    \left\|\sqrt{s(1-s)}e^{\left[\frac{\gamma\alpha\beta}
    {(\alpha s
    +\beta(1-s))^2}+
    \frac{(\alpha-\beta)a}{4(a^2+b^2)(\alpha s+\beta(1-s))}\right]|y|^2}\frac{|y|\sqrt{\alpha\beta}}
    {\alpha s+\beta(1-s)}u
    \right\|_{L^2(\R^n\times[0,1])}.
  \end{align}
\end{corollary}
We omit here the details of the proof of the previous Corollary, which are straightforward after Lemma \ref{lem:appell}.

\subsection{Logarithmic convexity}\label{subsec:logconv}
We now pass to study, from an abstract point of view, the evolution of weighted solutions to \eqref{eq:1} with gaussian weights.

\begin{lemma}\label{lem:A1}
  Let $u=u(x,t):\R^{n+1}\to\C$ be a solution to \eqref{eq:1}
where $a,b\in\R$, $A=A(x,t):\R^{n+1}\to\R^n$, $V,F:\R^{n+1}\to\C$, and denote by $v:=e^\varphi u$, with $\varphi=\varphi(x,t):\R^{n+1}\to\R$. Then $v$ solves
\begin{equation}\label{eq:A2}
  \partial_tv=\left(\mathcal S+\mathcal A\right)v
  +(a+ib)\left(V(x,t)v+e^{\varphi}F(x,t)\right),
\end{equation}
where
\begin{align}
\label{eq:S}
  \mathcal S
  &
  =
  a\left(\Delta_A+|\nabla_x\varphi|^2\right)
  -ib\left(\Delta_x\varphi+2\nabla_x\varphi\cdot\nabla_A\right)
  +\varphi_t
  \\
  \label{eq:A}
  \mathcal A
  &
  =
  ib
  \left(\Delta_A+|\nabla_x\varphi|^2\right)
  -a\left(\Delta_x\varphi+2\nabla_x\varphi\cdot\nabla_A\right).
\end{align}
In addition, the following identities hold:
\begin{align}
  \label{eq:St}
  \mathcal S_t
  &
  =2a\left(\Im A_t\cdot\nabla_A
  +\nabla_x\varphi\cdot\nabla_x\varphi_t\right)
  +2b\left(\Im\nabla_x\varphi_t\cdot\nabla_A-\nabla_x\varphi
  \cdot A_t\right) +\varphi_{tt}
  \\
  \label{eq:commutator}
  \int_{\R^n}[\mathcal S,\mathcal A]f\,\overline f\,dx
  &
  =(a^2+b^2)\left(4\int_{\R^n}\nabla_Af\cdot D^2_x\varphi\overline{\nabla_Af}\,dx
  -\int_{\R^n}|f|^2\Delta^2_x\varphi\,dx\right.
  \\
  \nonumber
  & \ \ \ \left.
  +4\int_{\R^n}|f|^2
  \nabla_x\varphi\cdot D^2_x\varphi\nabla_x\varphi\,dx
  -4\Im\int_{\R^n}f(\nabla_x\varphi)^t B\cdot\overline{\nabla_Af}\,dx
  \right)
  \\
  \nonumber
  & \ \ \
  +2b\Im\int_{\R^n}\overline{f}\nabla_x\varphi_t\cdot\nabla_A f\,dx+2a\int_{\R^n}|f|^2\nabla_x\varphi\cdot
  \nabla_x\varphi_t
  \,dx,
\end{align}
where $\mathcal S_t:=(\partial_t\mathcal S)$, $(D^2_x\varphi)_{jk}=\frac{\partial^2\varphi}
{\partial x_j\partial x_k}$, $\Delta^2_x\varphi:=\Delta_x(\Delta_x\varphi)$, $B=DA-DA^t$, $B_{jk}=A^k_j-A^j_k$ and $[\mathcal S,\mathcal A]=\mathcal S\mathcal A-\mathcal A\mathcal S$ denotes the commutator between $\mathcal S$ and $\mathcal A$.
\end{lemma}
Notice that $\mathcal S$ is a symmetric operator and $\mathcal A$ is skew-symmetric, with respect to the inner product in $L^2$.
The proof of Lemma \ref{lem:A1} is based on explicit computations and will be omitted. We mention the paper \cite{FV} for the computation of $[\Delta_A,\Delta_x\varphi+2\nabla_x\varphi\cdot\nabla_A]$, which is the only term in $[\mathcal S,\mathcal A]$ one has to compute with a bit of care.

We now prove a dissipation result for equation \eqref{eq:1}, which depends on the fact that $a>0$, and which permits to justify the proofs of the results in the sequel.

\begin{lemma}\label{lem:dissipation}
  Let $-\Delta_A$ be self-adjoint in $L^2$ and let $u\in L^\infty\left([0,1];L^2(\R^n)\right)\cap
  L^2\left([0,1]; H^1(\R^n)\right)$ be a solution to
  \begin{equation}\label{eq:A3}
  \partial_tu=(a+ib)\left(\Delta_Au+V(x,t)u+F(x,t)\right),
\end{equation}
in $\R^n\times[0,1]$, with $a>0$, $b\in\R$, $A=A(x):\R^n\to\R^n$, and $V,F:\R^{n+1}\to\C$.
Then, for any $\gamma\geq0$, $T\in[0,1]$, we have
\begin{align}\label{eq:dissipation}
  &
  e^{-M_T}\left\|e^{\frac{\gamma a}{a+4\gamma(a^2+b^2)T}|\cdot|^2}u(\cdot,T)\right\|
  _{L^2}
  \\
  \nonumber
  & \ \ \
  \leq
  \|e^{\gamma|\cdot|^2}u(\cdot,0)\|_{L^2}
  +\sqrt{a^2+b^2}\left\|e^{\frac{\gamma a}{a+4\gamma(a^2+b^2)t}|x|^2}F(x,t)\right\|
  _{L^1([0,T];L^2(\R^n))},
\end{align}
with $M_T:=\|a(\Re V)^+-b\Im V\|_{L^1([0,T];L^\infty(\R^n))}$, $(\Re V)^+$ being the positive part of $\Re V$.
\end{lemma}

\begin{proof}

 The proof is based on a standard energy method. First notice that, since $-\Delta_A$ is self-adjoint, solutions $u\in L^\infty\left([0,1];L^2(\R^n)\right)\cap
  L^2\left([0,1]; H^1(\R^n)\right)$ to \eqref{eq:A3} do exist by means of the Duhamel principle.

Let $v=e^{\varphi(x,t)}u$, satisfying
 \eqref{eq:A2} by Lemma \ref{lem:A1}.
 Formally, multiplying \eqref{eq:A2} by $\overline v$, integrating in $dx$ and taking the real parts, we obtain by \eqref{eq:S}, \eqref{eq:A} that
  \begin{align}\label{eq:0}
    \frac12\frac{d}{dt}\|v\|_{L^2}^2
   &
    =
    \Re\int\mathcal Sv\,\overline v\,dx
    +\Re\left\{(a+ib)\int\left(|v|^2V+e^\varphi F\overline v\right)\,dx\right\}
    \\
    \nonumber
    &
    =
    -a\int|\nabla_Av|^2\,dx
    +a\int|\nabla_x\varphi|^2|v|^2\,dx
    +\int\varphi_t|v|^2\,dx
    \\
    &
    \nonumber
    \ \ \
    +2b\Im\int\overline v\nabla_x\varphi\cdot\nabla_A v\,dx
    +\Re(a+ib)\int\left(|v|^2V+e^\varphi F\overline v\right)\,dx.
  \end{align}
  We can easily estimate
  \begin{align}\label{eq:V}
    \Re(a+ib)\int|v|^2V\,dx
    &
   \leq
    \left\|a(\Re V)^+-b(\Im V)\right\|_{L^\infty}\|v\|_{L^2}^2
    \\
    \label{eq:F}
    \Re(a+ib)\int e^\varphi F\overline v\,dx
    &
    \leq
    \sqrt{a^2+b^2}\left\|e^\varphi F\right\|_{L^2}
    \|v\|_{L^2}.
  \end{align}
  Analogously, by Cauchy-Schwartz we have
  \begin{equation}\label{eq:CS}
    2b\Im\int\overline v\nabla_x\varphi\cdot\nabla_A v\,dx
    \leq
    a\int|\nabla_Av|^2\,dx+\frac{b^2}{a}\int|\nabla_x\varphi|^2
    |v|^2\,dx;
  \end{equation}
  as a consequence, by \eqref{eq:0} and \eqref{eq:CS} we obtain
  \begin{equation}\label{eq;re}
    \Re\int\mathcal Sv\,\overline v\,dx
    \leq
    \int\left\{\left(a+\frac{b^2}{a}\right)|\nabla_x\varphi|^2
    +\varphi_t\right\}|v|^2\,dx,
  \end{equation}
  and the choice
  \begin{equation}\label{eq:phi}
   \varphi(x,t)=\frac{\gamma a}{a+4\gamma(a^2+b^2)t}|x|^2
    \qquad
    \Rightarrow
    \qquad
    \varphi_t(x,t)=-\left(a+\frac{b^2}{a}\right)|\nabla_x\varphi|^2
  \end{equation}
  gives in turn that
  \begin{equation}\label{eq:Res}
    \Re\int\mathcal Sv\,\overline v\,dx
    \leq
   0.
  \end{equation}
  By \eqref{eq:0}, \eqref{eq:V}, \eqref{eq:F}, \eqref{eq:Res}, with the choice \eqref{eq:phi}, we finally obtain
  \begin{align*}
  &
    \frac{d}{dt}\|v(\cdot, t)\|_{L^2}^2
    \\
    &
    \leq
    2\left\|a(\Re V)^+-b(\Im V)\right\|_{L^\infty}\|v(\cdot, t)\|_{L^2}^2
    +
    2\sqrt{a^2+b^2}\left\|e^\varphi F\right\|_{L^2}
    \|v(\cdot, t)\|_{L^2},
  \end{align*}
  which implies \eqref{eq:dissipation}.

  In order to make the previous argument rigorous, since the exponentially weighted $L^2$-norms involved in the integration by parts are not finite in principle, it is sufficient to work with truncated and mollified weights of the following form:
  \begin{equation*}
    \varphi_R(x,t)
    =
    \begin{cases}
      \varphi(x,t),
      \qquad
      \text{if }|x|<R
      \\
      \varphi(R,t),
      \qquad
      \text{if }|x|\geq R,
    \end{cases}
    \qquad
    \varphi_{R,\epsilon}:=
    (\theta_{\epsilon}\ast\varphi_R)(x),
  \end{equation*}
  $\theta_\epsilon(x)$ being a radial mollifier.
  Then the result can be obtained by performing the same computation as above and
  then letting $\epsilon$ go to 0 and $R$ to $\infty$; we omit straightforward details.
  \end{proof}
  \begin{remark}\label{rem:dissipation}
    Notice that the dissipation estimate \eqref{eq:dissipation} has been proved for stationary magnetic potentials $A=A(x)$. In the time-dependent case $A=A(x,t)$, the same result would require some additional assumptions on the time derivative $A_t$, since we need the self-adjointness property, which at this level seem quite artificial.  \end{remark}

The next result, proved by Escauriaza, Kenig, Ponce and Vega in \cite{EKPV1,EKPV2}, is the abstract core of Theorem \ref{thm:3}. It is concerned with the connection between the positivity of $\mathcal S_t+[\mathcal S,\mathcal A]$ and the logarithmic convexity of weighted $L^2$-norms with gaussian weights.

\begin{lemma}[logarithmic convexity]
\label{lem:logconv}
  Let $\mathcal S$ be a symmetric operator, $\mathcal A$ a skew-symmetric one, both with coefficients depending on $x$ and $t$, $f=f(x,t):\R^{n+1}\to\C$ be a sufficiently regular function, $G$ a positive function, and denote by
  \begin{equation}\label{eq:HDN}
    H(t)=\int_{\R^n}|f|^2\,dx.
  \end{equation}
  Assume that
  \begin{equation}\label{eq:logconv1}
    \left|\partial_tf-(\mathcal S+\mathcal A)f\right|
    \leq M_1|f|+G\ \ \text{in }\R^n\times[0,1],
    \qquad
    \mathcal S_t+[\mathcal S,\mathcal A]\geq-M_0,
  \end{equation}
  for some $M_0,M_1\geq0$ and
  \begin{equation}\label{eq:logconv2}
    M_2:=\sup_{t\in[0,1]}\frac{\|G(t)\|_{L^2}}
    {\|f(t)\|_{L^2}}<\infty.
  \end{equation}
  Then the function $\psi(t):=\log H(t)$ is convex in $[0,1]$. In particular, if
  \begin{equation}\label{eq:rigorous}
    H(0)<\infty
    \quad
    \Rightarrow
    \quad
    H(t)<\infty\ \text{for any }t\in[0,1],
  \end{equation}
  then there exist a universal constant $N\geq0$ such that
  \begin{equation}\label{eq:logconv}
    H(t)\leq e^{N(M_0+M_1+M_2+M_1^2+M_2^2)}
    H(0)^{1-t}H(1)^t,
  \end{equation}
  for any $t\in[0,1]$.
\end{lemma}
\begin{remark}\label{rem:rigorous}
The proof of Lemma \ref{lem:logconv} is based on the computation of the time derivatives $\dot H(t),\ddot H(t)$. An explicit (formal) computation gives
\begin{align}\label{eq:2ekpv2}
    \frac{d^2}{dt^2}H(t)
    &
    =2\partial_t\Re\int
    \overline v(\partial_t-\mathcal S-\mathcal A)v\,dx
    +2\int\overline v
    (\mathcal S_t+[\mathcal S,A])v\,dx
    \\
    \nonumber
    & \ \ \
    +\left\|\partial_tv-\mathcal Av+\mathcal Sv\right\|_{L^2}^2
    -\left\|\partial_tv-\mathcal Av-\mathcal Sv\right\|_{L^2}^2.
\end{align}
This, together with the computation of the first derivative $\dot H(t)$, shows that, under conditions \eqref{eq:logconv1}, \eqref{eq:logconv2}, the second derivative $\frac{d^2}{dt^2}\log(H(t))$ is positive. Assumption \eqref{eq:rigorous} is then the essential information one needs in order to conclude the convexity inequality \eqref{eq:logconv}. The validity of condition \eqref{eq:rigorous} depends on an energy estimate of the type \eqref{eq:dissipation} and needs to be checked each time when Lemma \ref{lem:logconv} is applied to explicit operators $\mathcal S,\mathcal A$, as we see in the following results.

The proof of Lemma \ref{lem:logconv} can be found in \cite{EKPV1, EKPV2}.
\end{remark}

We can finally prove the main results of this section.
\begin{lemma}\label{lem:3}
  Let $u\in L^\infty\left([0,1];L^2(\R^n)\right)\cap
  L^2\left([0,1]; H^1(\R^n)\right)$ be a solution to
  \begin{equation}\label{eq:A4}
  \partial_tu=(a+ib)\left(\Delta_Au+V(x,t)u+F(x,t)\right),
\end{equation}
in $\R^n\times[0,1]$, with $a>0$, $b\in\R$, $A=A(x,t):\R^{n+1}\to\R^n$, and $V,F:\R^{n+1}\to\C$.
Assume that
\begin{equation}\label{eq:gauge}
  x\cdot A(x)\equiv0\equiv x\cdot A_t(x).
\end{equation}
Moreover, let $\gamma>0$ and assume that
\begin{equation}\label{eq:3V}
  \sup_{t\in[0,1]}\|V(t,\cdot)\|_{L^\infty}:=M_1<\infty,
  \qquad
  \sup_{t\in[0,1]}
  \frac{\left\|e^{\gamma|\cdot|^2}F(\cdot,t)\right\|
  _{L^2}}{\|u(\cdot,t)\|_{L^2}}:=M_2<\infty;
\end{equation}
in addition, denote by $B=B(x,t)=D_xA-D_xA^t$ and assume
\begin{equation}\label{eq:3A}
  \frac{1}{\gamma}\sup_{t\in[0,1]}
  \left\|A_t(\cdot,t)\right\|_{L^\infty}^2
  +4\gamma(a^2+b^2)\sup_{t\in[0,1]}\|x^tB(\cdot,t)\|
  _{L^\infty}^2
  :=M_A<\infty.
\end{equation}
Finally, assume
\begin{equation}\label{eq:3decay}
  \left\|e^{\gamma|\cdot|^2}u(\cdot,0)\right\|_{L^2}
  +\left\|e^{\gamma|\cdot|^2}u(\cdot,1)\right\|_{L^2}
  <\infty;
\end{equation}
finally, define $H(t)=\left\|e^{\gamma|\cdot|^2}u(\cdot,t)\right\|_{L^2}$
and assume that \eqref{eq:rigorous} holds. Then, $H(t)$  is finite and logarithmically convex in $[0,1]$; in particular, there exists a constant $N=N(\gamma,a,b)$ such that
\begin{equation}\label{eq:3tesi}
  H(t)\leq
  e^{N\left[M_A+\sqrt{a^2+b^2}(M_1+M_2)+(a^2+b^2)
  (M_1^2+M_2^2)\right]}
  \left\|e^{\gamma|\cdot|^2}u(\cdot,0)\right\|_{L^2}^{1-t}
  \left\|e^{\gamma|\cdot|^2}u(\cdot,1)\right\|_{L^2}^t,
\end{equation}
for any $t\in[0,1]$.
\end{lemma}

\begin{remark}
  Before the proof, we need another remark about condition \eqref{eq:rigorous} in the statement.
  The result ensuring, in concrete situations, the validity of \eqref{eq:rigorous}, is Lemma \ref{lem:dissipation}. Notice that in the statement of Lemma \ref{lem:3} we work with magnetic potentials $A=A(x,t)$ which possibly depend on time, while the time dependence is not permitted in Lemma \ref{lem:dissipation}.
  In fact, as we see in the next section, in the proof of Theorem \ref{thm:3}, after applying the Appell transformation, a natural time dependence of the magnetic potential appears. On the other hand, condition \eqref{eq:rigorous} will hold in the the next section as a heritage of the same property before the Appell transformation, and no additional assumptions on $\partial_tA$ will be needed. This explains why we prefer to assume \eqref{eq:rigorous} in the previous statement without giving explicit conditions under which it is satisfied.
\end{remark}
\begin{proof}[Proof of Lemma \ref{lem:3}]
  We need to check that Lemma \ref{lem:logconv} is applicable.

  Denote again by $v=e^{\varphi(x,t)}u$, with $\varphi(x,t)=\varphi(x):=\gamma|x|^2$. By Lemma \ref{lem:A1}, $v$ satisfies
  \begin{equation*}
    \partial_tv=\mathcal Sv+\mathcal Av
    +(a+ib)\left(V(x,t)v+e^{\varphi}F\right),
  \end{equation*}
  where $\mathcal S$ and $\mathcal A$ are given by \eqref{eq:S}, \eqref{eq:A}, respectively.
  We can estimate
  \begin{equation}\label{eq:G}
    \left|\partial_tv-(\mathcal S+\mathcal A)v\right|
    \leq
    \sqrt{a^2+b^2}\left(M_1|v|+e^{\varphi}|F|\right),
  \end{equation}
  which proves the first of the two conditions in \eqref{eq:logconv1}, with $G:=\sqrt{a^2+b^2}e^\varphi|F|$.
  Hence we just need to check the second condition in \eqref{eq:logconv1}.
  By formulas \eqref{eq:St} and \eqref{eq:commutator} with the choice $\varphi(x)=\gamma|x|^2$ we obtain
  \begin{align}\label{eq:stsa}
    &
    \int\overline v\left(\mathcal S_t+[\mathcal S,\mathcal A]\right)v\,dx
    =2a\Im\int\overline vA_t\cdot\nabla_Av\,dx
    -4b\gamma\int|v|^2x\cdot A_t\,dx
    \\
    \nonumber
    & \ \
    +(a^2+b^2)
    \left\{8\gamma\int |\nabla_Av|^2\,dx
    +8\gamma\Im\int
    \overline vx^tB\cdot\nabla_Av\,dx
    +32\gamma^3\int|v|^2|x|^2\,dx\right\}.
  \end{align}
  The second term at the right-hand side of \eqref{eq:stsa} vanishes, due to \eqref{eq:gauge}.
  By Cauchy-Schwartz, we can estimate
  \begin{align}\label{eq:4}
    \left|2a\Im\int\overline vA_t\cdot\nabla_Av\,dx\right|
    &
    \leq
    \frac{1}{\gamma}\int|A_t|^2|v|^2\,dx+\gamma a^2\int\left|\nabla_Av\right|^2\,dx
    \\
    \label{eq:6}
    \left|8\gamma(a^2+b^2)\Im\int
    \overline vx^tB\cdot\nabla_Av\,dx\right|
    &
    \leq
    4\gamma(a^2+b^2)\int|x^tB|^2|v|^2\,dx
    \\
    \nonumber
    &
    +4\gamma(a^2+b^2)\int|\nabla_Av|^2\,dx;
  \end{align}
  by \eqref{eq:stsa}, \eqref{eq:4}, \eqref{eq:6} it turns out that
  \begin{align}\label{eq:astast}
    &
    \int\overline v\left(\mathcal S_t+[\mathcal S,\mathcal A]\right)v\,dx
    \geq
    3\gamma(a^2+b^2)\int|\nabla_Av|^2\,dx
    +32\gamma^3(a^2+b^2)\int|v|^2|x|^2
    \\
    \nonumber
    & \ \ \ \ \ \
    -\left(\frac{1}{\gamma}
    \sup_{t\in[0,1]}\|A_t\|_{L^\infty}^2
    +4\gamma(a^2+b^2)\sup_{t\in[0,1]}\|x^tB\|_{L^\infty}^2
    \right)\int|v|^2\,dx.
  \end{align}
  Neglecting the positive terms in the last inequality, we have proved that
  \begin{equation}\label{eq:7}
    \mathcal S_t+[\mathcal S,\mathcal A]
    \geq
    -\frac{1}{\gamma}
    \sup_{t\in[0,1]}\|A_t\|_{L^\infty}^2
    -4\gamma(a^2+b^2)\sup_{t\in[0,1]}\|x^tB\|_{L^\infty}^2
    =-M_A.
  \end{equation}
  In addition, we have
  \begin{equation}\label{eq:8}
    \sup_{t\in[0,1]}
    \frac{\sqrt{a^2+b^2}\left\|e^{\gamma|\cdot|^2}
    F(\cdot,t)\right\|_{L^2}}{\|v(\cdot,t)\|_{L^2}}
    \leq
    \sqrt{a^2+b^2}M_2.
  \end{equation}
  The thesis now follows by Lemma \ref{lem:logconv}.

In order to obtain a completely rigorous proof of Lemma \ref{lem:3} we need a last remark. The positive dissipation $a>0$ provides the sufficient interior regularity for Lemma \ref{lem:dissipation} to hold. In the next section, when we apply Lemma \ref{lem:3} to a concrete situation, in order to justify all the above computations we need to work with the following multipliers.
  Given $a>0$ and $\rho\in(0,1)$, define
  \begin{equation*}
    \varphi_a(x)
    =
    \begin{cases}
     \gamma|x|^2,
     \qquad\qquad
     \text{if }
      |x|<1
     \\
     \gamma\frac{2|x|^{2-a}-a}{2-a}
     \qquad\,
     \text{if }
     |x|\geq1
    \end{cases}
  \end{equation*}
and replace $\varphi=\gamma|x|^2$ by $\varphi_{a,\rho}=\theta_\rho\star\varphi_a$, being $\theta_\rho$ a smooth delta-sequence. One can easily check that all the above computations are then justified as a limit when $a,\rho\to0$. See \cite{EKPV2} for further details.
\end{proof}

In an analogous way, we prove the following result:
\begin{lemma}\label{lem:4}
  Under the same assumptions as in Lemma \ref{lem:3}, there exists a constant $N=N\left(\frac{1}{\gamma},\frac{1}{a^2+b^2}\right)>0$ such that
  \begin{align}
    \label{eq:9}
    &
    \left\|\sqrt{t(1-t)}e^{\gamma|x|^2}
    \nabla_Au(x,t)\right\|_{L^2(\R^n\times[0,1])}
    +\gamma\left\|\sqrt{t(1-t)}e^{\gamma|x|^2}
    |x|u(x,t)\right\|_{L^2(\R^n\times[0,1])}
    \\
    \nonumber
    &
    \ \ \
    \leq
    N\left[(M_1+\sqrt{M_A}+1)\sup_{t\in[0,1]}
    \left\|e^{\gamma|\cdot|^2}
    u(\cdot,t)\right\|_{L^2}+
    \sup_{t\in[0,1]}
    \left\|e^{\gamma|\cdot|^2}
    F(\cdot,t)\right\|_{L^2}\right].
  \end{align}
\end{lemma}
\begin{proof}
  Denote again by $v=e^{\gamma|x|^2}u$; we can hence write
  \begin{equation*}
    \nabla_Au=-2\gamma xe^{-\gamma|x|^2}v+
    e^{-\gamma|x|^2}\nabla_Av.
  \end{equation*}
  Consequently, we can estimate
  \begin{align}
  \label{eq:1ekpv}
  &
    \left\|\sqrt{t(1-t)}e^{\gamma|x|^2}
    \nabla_Au(x,t)\right\|_{L^2(\R^n\times[0,1])}
    +\gamma\left\|\sqrt{t(1-t)}e^{\gamma|x|^2}
    |x|u(x,t)\right\|_{L^2(\R^n\times[0,1])}
    \\
    \nonumber
    &
    \ \
    \leq
    3\gamma\left\|\sqrt{t(1-t)}
    |x|v(x,t)\right\|_{L^2(\R^n\times[0,1])}
    +\left\|\sqrt{t(1-t)}
    \nabla_Av(x,t)\right\|_{L^2(\R^n\times[0,1])}.
  \end{align}
  By \eqref{eq:2ekpv2}, we easily estimate
  \begin{align}
  \label{eq:2ekpv}
    \frac{d^2}{dt^2}H(t)
        &
    \geq
    2\partial_t\Re\int
    \overline v(\partial_t-\mathcal S-\mathcal A)v\,dx
    +2\int\overline v
    (\mathcal S_t+[\mathcal S,A])v\,dx
    \\
    \nonumber
    & \ \ \
    -\left\|\partial_tv-\mathcal Av-\mathcal Sv\right\|_{L^2_x}^2.
  \end{align}
  On the other hand, integrating twice by parts we get
  \begin{equation}\label{eq:3ekpv}
    \int_0^1
    t(1-t)\frac{d^2}{dt^2}H(t)\,dt
    =
    H(1)+H(0)-2\int_0^1H(t)\,dt
    \leq
    2\sup_{t\in[0,1]}\|v(\cdot,t)\|_{L^2}^2,
  \end{equation}
  since $H(t)\geq0$.
  Integrating by parts and applying Cauchy-Schwartz and estimate \eqref{eq:G}, we obtain
  \begin{align}
    \label{eq:4ekpv}
    &
    2\int_0^1\int
    t(1-t)\partial_t\Re\overline v(\partial_t-\mathcal S
    -\mathcal A)v\,dx\,dt
    \\
    \nonumber
    & \ \
    =
    -2\int_0^1\int(1-2t)\Re\overline v(\partial_t-\mathcal S
    -\mathcal A)v\,dx\,dt
    \\
    \nonumber
    & \ \
    \geq
    -
    \left(\sup_{t\in[0,1]}\left\|
    \partial_tv-\mathcal Sv-\mathcal Av\right\|_{L^2}^2
    +\sup_{t\in[0,1]}\|v(\cdot,t)\|_{L^2}^2\right)
    \\
    \nonumber
    & \ \
    \geq
    -\frac12\left\{
    \left[(a^2+b^2)M_1^2+1\right]
    \sup_{t\in[0,1]}\|v(\cdot,t)\|_{L^2}^2
    +(a^2+b^2)\sup_{t\in[0,1]}\left\|
    e^{\gamma|\cdot|^2}F(\cdot,t)\right\|_{L^2}^2\right\}.
  \end{align}
  On the other hand, by \eqref{eq:astast} we get
  \begin{align}
    \label{eq:5ekpv}
    &
    2\int_0^1\int
    t(1-t)
    \overline v(\mathcal S_t+[\mathcal S,\mathcal A])v
    \,dx\,dt
    \geq  -\frac{M_A}{3}\sup_{t\in[0,1]}\left\|v(\cdot,t)\right\|_{L^2}^2
    \\
    \nonumber
    & \ \
    +
    2\gamma(a^2+b^2)
    \left\{
    \left\|\sqrt{t(1-t)}\nabla_Av\right\|
    _{L^2(\R^n\times[0,1])}^2+
    \gamma^2\left\|\sqrt{t(1-t)}|x|v\right\|
    _{L^2(\R^n\times[0,1])}^2\right\}
   \end{align}
  while by \eqref{eq:G} we conclude that
  \begin{align}
    \label{eq:6ekpv}
    &
    -\int_0^1t(1-t)
    \left\|\partial_tv-\mathcal Sv-\mathcal Av\right\|
    _{L^2}^2\,dt
    \\
    \nonumber
    & \ \
    \geq
    -\sup_{t\in[0,1]}
    \left\|\partial_tv-\mathcal Sv-\mathcal Av\right\|
    _{L^2}^2\int_0^1t(1-t)\,dt
    \\
    \nonumber
    & \ \
    \geq
    -\frac16\left(a^2+b^2\right)
    \left\{
    M_1^2\sup_{t\in[0,1]}\|v(\cdot,t)\|_{L^2}^2
    +\sup_{t\in[0,1]}\left\|e^{\gamma|\cdot|^2}F(\cdot,t)
    \right\|_{L^2}^2\right\}.
  \end{align}
  Collecting \eqref{eq:2ekpv}, \eqref{eq:3ekpv}, \eqref{eq:4ekpv}, \eqref{eq:5ekpv}, \eqref{eq:6ekpv} we have
  \begin{align*}
    &
    \left\|\sqrt{t(1-t)}\nabla_Av\right\|
    _{L^2(\R^n\times[0,1])}^2+
    \gamma^2\left\|\sqrt{t(1-t)}|x|v\right\|
    _{L^2(\R^n\times[0,1])}^2
    \\
    & \ \
    \leq
    \left[\frac{M^2_1}{3\gamma}
    +\frac{15+2M_A}{12\gamma(a^2+b^2)}\right]
    \sup_{t\in[0,1]}\|v(\cdot,t)\|_{L^2}^2
    +\frac{1}{3\gamma}\sup_{t\in[0,1]}\left\|
    e^{\gamma|\cdot|^2}F(\cdot,t)\right\|_{L^2}^2,
  \end{align*}
  which, together with \eqref{eq:1ekpv}, proves the claim \eqref{eq:9}.

  Also in this case, the proof can be made rigorous by a quite standard argument in the spirit of the one in Lemma \ref{lem:3}.
\end{proof}

All the tools we need to prove Theorem \ref{thm:3} are now ready.

\section{Proof of Theorem \ref{thm:3}}

For the proof of Theorem \ref{thm:3}, we now put together the informations we got in the previous Section. It is sufficient to prove the result in the case $\alpha<\beta$; for the proof in the case $\alpha>\beta$ replace $u(x,t)$ by $\overline u(x,1-t)$, while in the case $\alpha=\beta$ the proof essentially reduces to Lemma \ref{lem:3} and \ref{lem:4} (see Remark \ref{rem:alfabeta} below).
Therefore, from now on we assume
\begin{equation*}
\alpha<\beta.
\end{equation*}
We divide the proof of Theorem \ref{thm:3} into four steps.

\subsection{Step I: the gauge reduction}
Thanks to assumption \eqref{eq:gauge} and Lemma \ref{lem:cronstrom1}, it is now sufficient to prove
Theorem \ref{thm:3} for the function $\widetilde u=e^{i\varphi}u$, where $\varphi$ is the gauge change
defined in \eqref{eq:varphi}. The new potential is $\widetilde A$, defined in \eqref{eq:cronstrom2}. By
abuse of notations, we will skip the tildes; hence, from now on, the additional (and not restrictive)
assumption
\begin{equation}\label{eq:gauge3}
  x\cdot A \equiv 0
\end{equation}
holds, together with the identities \eqref{eq:cronstrom2}, \eqref{eq:cronstrom3}, which in our new
notations read as
\begin{equation}\label{eq:identities}
  A(x)=-\int_0^1\Psi(sx)\,ds;
  \qquad
  x^tDA(x)=
  -\Psi(x)+\int_0^1\Psi(sx)\,ds,
\end{equation}
with $\Psi(x)=x^tB(x)=x^t(DA(x)-DA^t(x))$, which also gives
\begin{equation}\label{eq:gauge32}
  x\cdot x^tDA\equiv0.
\end{equation}
 In particular, \eqref{eq:assA3} and \eqref{eq:identities} also imply that
\begin{equation}\label{eq:assA3new}
  \|A\|_{L^\infty}^2+\|x^tDA\|_{L^\infty}^2
  +\|x^tB\|_{L^\infty}^2\leq M_A.
\end{equation}

\subsection{Step II: the heat regularization}
We now regularize equation \eqref{eq:main3} adding a small dissipation term. Denote by
\begin{equation*}
  H_A:=-\Delta_A-V_1
\end{equation*}
and rewrite equation \eqref{eq:main3} as
\begin{equation}
  \partial_tu=
  -i(H_Au-F(x,t)),
  \qquad
  F(x,t):=V_2(x,t)u.
\end{equation}
Since $H_A$ is self-adjoint by Proposition \ref{prop:self}, we can define, by Spectral Theorem, the mixed flow $e^{(\epsilon+i)tH_A}$, for any $\epsilon>0$. This gives, by parabolic regularity, the function
\begin{equation}\label{eq:uepsilon}
  u_\epsilon(\cdot,t):=e^{(\epsilon+i)tH_A}u(\cdot,0)=
  e^{\epsilon tH_A}u(t)\in
  L^\infty([0,1];L^2(\R^n))\cap L^2([0,1];H^1(\R^n)),
\end{equation}
solving (uniquely) the equation
\begin{equation}\label{eq:uepsiloneq}
  \begin{cases}
  \partial_tu_\epsilon
  =
  (\epsilon+i)\left(\Delta_Au_\epsilon
  +V_1(x)u_\epsilon+F_\epsilon(x,t)\right),
  \\
  u_\epsilon(0)=u(0),
  \end{cases}
\end{equation}
with $F_\epsilon(\cdot,s):=\frac{i}{\epsilon+i}
  e^{\epsilon tH_A}\left(V_2u(\cdot,s)\right)$ (see e.g. \cite{P,S}).
The positive dissipation now permits to apply  Lemma \ref{lem:dissipation}, which is useful in the sequel to make rigorous the applications of Lemma \ref{lem:3}.
We can now prove the following simple result.
\begin{lemma}\label{lem:B}
  Denote by
  \begin{equation}\label{eq:alphabetaepsilon}
    \alpha_\epsilon^2=
    \alpha^2+4\epsilon
    \qquad
    \beta_\epsilon^2=\beta^2+4\epsilon.
  \end{equation}
  The function $u_\epsilon$ defined in \eqref{eq:uepsilon} satisfies the following inequalities:
  \begin{align}
    \label{eq:be}
    \left\|e^{\frac{|\cdot|^2}{\beta_\epsilon^2}}
    u_\epsilon(\cdot,0)\right\|_{L^2}
    &
    \leq
    \left\|e^{\frac{|\cdot|^2}{\beta^2}}
    u(\cdot,0)\right\|_{L^2}
    \\
    \label{eq:ae}
    \left\|e^{\frac{|\cdot|^2}{\alpha_\epsilon^2}}
    u_\epsilon(\cdot,1)\right\|_{L^2}
    &
    \leq
    e^{\epsilon\|V_1\|_{L^\infty}}
    \left\|e^{\frac{|\cdot|^2}{\alpha^2}}
    u(\cdot,1)\right\|_{L^2}
    \\
    \label{eq:te}
    \left\|u_\epsilon(\cdot,t)\right\|_{L^2}
    &
    \leq
    e^{\epsilon\|V_1\|_{L^\infty}}
    \|u(\cdot,t)\|_{L^2}
    \\
    \label{eq:tf}
    \left\|F_\epsilon(\cdot,t)\right\|_{L^2}
    &
    \leq
    e^{\epsilon\|V_1\|_{L^\infty}}
    \|V_2\|_{L^\infty}\|u(\cdot,t)\|_{L^2}
    \\
    \label{eq:f}
    \left\|e^{\frac{|\cdot|^2}{(\alpha_\epsilon t
    +\beta_\epsilon(1-t))^2}}F_\epsilon(\cdot,t)\right\|_{L^2}
    &
    \leq
    e^{\epsilon\|V_1\|_{L^\infty}}
    \left\|e^{\frac{|\cdot|^2}{(\alpha_\epsilon t
    +\beta_\epsilon(1-t))^2}}V_2(\cdot,t)\right\|_{L^\infty}
    \|u(\cdot,t)\|_{L^2},
  \end{align}
  for any $t\in[0,1]$.
\end{lemma}
\begin{proof}
  Inequality \eqref{eq:be} is immediate.

  In order to prove \eqref{eq:ae}, let us introduce the function $w(\cdot,t):=e^{-\epsilon tH_A}u(\cdot,1)$, solving the equation
  \begin{equation*}
    \partial_tw=-\epsilon H_Aw=\epsilon(\Delta_Aw+V_1w).
  \end{equation*}
  Then \eqref{eq:ae} follows applying inequality \eqref{eq:dissipation} to $w$, with $\gamma:=\frac1{\alpha^2}$ and $T=1$.

  To prove \eqref{eq:te} write $u_\epsilon(\cdot,t):=e^{\epsilon tH_A}u(t)$ and apply again \eqref{eq:dissipation}, with $\gamma=0$ and $T=t$.

  For the proof of \eqref{eq:tf}, introduce the function $w(\cdot,t):=e^{\epsilon tH_A}(V_2u(\cdot,t))$ and apply again \eqref{eq:dissipation}, with $\gamma=0$, $T=t$. Finally, by the application of inequality \eqref{eq:dissipation} to the same function, with $\gamma=\frac{1}{(\alpha t+\beta(1-t))^2}$ and $T=t$, the proof of \eqref{eq:f} easily follows.
\end{proof}

\subsection{Step III: the Appell transformation}\label{subsec:appellproof}

We now apply the Appell transformation to the function $u_\epsilon$. Let $\alpha_\epsilon,\beta_\epsilon$ be the same as in \eqref{eq:alphabetaepsilon} and define
\begin{align}
\label{eq:appellproof}
  &
  \widetilde u_\epsilon(x,t)
  :=
  \left(\frac{\sqrt{\alpha_\epsilon\beta_\epsilon}}
  {\alpha_\epsilon (1-t)+\beta_\epsilon t}\right)^{\frac n2}
  \times
  \\
  \nonumber
  &
  \times
  u_\epsilon\left(
  \frac{\sqrt{\alpha_\epsilon\beta_\epsilon}}
  {\alpha_\epsilon (1-t)+\beta_\epsilon t}x,
  \frac{\beta_\epsilon}
  {\alpha_\epsilon (1-t)+\beta_\epsilon t}t\right)
  e^{\frac{\alpha_\epsilon-\beta_\epsilon}
  {4(\epsilon+i)(\alpha_\epsilon (1-t)+\beta_\epsilon t)}
  |x|^2}.
\end{align}
Since $x\cdot A\equiv0$ due to step I, by Lemma \ref{lem:appell} we have that $\widetilde u_\epsilon$ solves
\begin{equation}\label{eq:2appellproof}
  \partial_t\widetilde u_\epsilon=(\epsilon+i)\left(\Delta_{\widetilde A_\epsilon}\widetilde u_\epsilon
  +
  \widetilde V_\epsilon(x,t)\widetilde u_\epsilon+\widetilde F_\epsilon(x,t)\right),
\end{equation}
where
\begin{align}
  \label{eq:Aappellproof}
  \widetilde A_\epsilon(x,t)
  &
  = \frac{\sqrt{\alpha_\epsilon
  \beta_\epsilon}}{\alpha_\epsilon(1-t)+\beta_\epsilon t}
  A\left(\frac{\sqrt{\alpha_\epsilon\beta
  _\epsilon}}{\alpha_\epsilon(1-t)+\beta_\epsilon t}x\right)
  \\
  \label{eq:Vappellproof}
  \widetilde V_\epsilon(x,t)
  &
  = \frac{\alpha_\epsilon\beta_\epsilon}{(\alpha_\epsilon
  (1-t)+\beta_\epsilon t)^2}
  V_1\left(\frac{\sqrt{\alpha_\epsilon
  \beta_\epsilon}}{\alpha_\epsilon(1-t)+\beta_\epsilon t}x\right)
  \\
  \label{eq:Fappellproof}
  \widetilde F_\epsilon(x,t)
  &
  = \left(\frac{\sqrt{\alpha_\epsilon
  \beta_\epsilon}}{\alpha_\epsilon
  (1-t)+\beta_\epsilon t}\right)^{\frac n2+2}
  \times
  \\
  \nonumber
  & \ \ \ \times
  F_\epsilon
  \left(\frac{\sqrt{\alpha_\epsilon\beta_\epsilon}}
  {\alpha_\epsilon(1-t)+\beta_\epsilon t}x,\frac{\beta_\epsilon}{\alpha_\epsilon
  (1-t)+\beta_\epsilon t}t\right)
  e^{\frac{(\alpha_\epsilon-\beta_\epsilon)|x|^2}{4(\epsilon+i)
  (\alpha_\epsilon(1-t)+\beta_\epsilon t)}}.
\end{align}
In addition, by Corollary \ref{cor:appell}, for any $\gamma\in\R$ we have
\begin{align}
    \label{eq:corappell1proof}
    &
    \left\|e^{\gamma|\cdot|^2}\widetilde u_\epsilon(\cdot,t)\right\|
    _{L^2}
    =\left\|e^{\left[\frac{\gamma\alpha_\epsilon
    \beta_\epsilon}{(\alpha_\epsilon s
    +\beta_\epsilon(1-s))^2}+
    \frac{(\alpha_\epsilon-\beta_\epsilon)
    \epsilon}{4(\epsilon^2+1)(\alpha_\epsilon s+\beta_\epsilon(1-s))}\right]|\cdot|^2}u_\epsilon
    (\cdot,s)\right\|_{L^2}
    \\
    \label{eq:corappell2proof}
    &
    \left\|e^{\gamma|\cdot|^2}\widetilde F_\epsilon(\cdot,t)\right\|
    _{L^2}
    =
    \frac{\alpha_\epsilon\beta_\epsilon}
    {(\alpha_\epsilon(1-t)+\beta_\epsilon t)^2}
    \times
    \\
    \nonumber
    &
    \qquad\qquad\qquad\ \ \ \
    \times
    \left\|e^{\left[\frac{\gamma\alpha_\epsilon
    \beta_\epsilon}{(\alpha_\epsilon s
    +\beta_\epsilon(1-s))^2}+
    \frac{(\alpha_\epsilon-\beta_\epsilon)
    \epsilon}{4(\epsilon^2+1)(\alpha_\epsilon s+\beta_\epsilon(1-s))}\right]|\cdot|^2}F_\epsilon
    (\cdot,s)\right\|_{L^2},
\end{align}
for $s=\frac{\beta_\epsilon t}{\alpha_\epsilon (1-t)+\beta_\epsilon t}$.

The goal is to apply Lemma \ref{lem:3} to the function $\widetilde u_\epsilon$. In order to do this,
we now need two more results regarding the evolution of the $L^2_x$-norms of $u$ and $\widetilde u_\epsilon$.
\begin{lemma}\label{lem:C1C2}
 Denote by
 \begin{equation}\label{eq:N1}
   N_1:=e^{\sup_{t\in[0,1]}\left\|
   \Im V_2(\cdot,t)\right\|_{L^\infty}}.
 \end{equation}
 The following inequalities hold
 \begin{align}
 \label{eq:m}
   \frac{1}{N_1}
   \|u(\cdot,0)\|_{L^2}
   &
   \leq
   \|u(\cdot,t)\|_{L^2}
   \leq
   N_1\|u(\cdot,0)\|_{L^2}
   \\
   \label{eq:m2}
   \frac{d}{dt}
   \left\|\widetilde u_{\epsilon}
   (\cdot,t)\right\|_{L^2}
   &
   \leq
   \epsilon\frac{\beta}{\alpha}
   e^{\epsilon\|V_1\|_{L^\infty}}
   N_1\|u(\cdot,0)\|_{L^2}
   \left(\|V_1\|_{L^\infty}+\sup_{t\in[0,1]}
   \|V_2(\cdot,t)\|_{L^\infty}\right),
 \end{align}
 for any $t\in[0,1]$, where $u$ is a solution to \eqref{eq:main3} and $\tilde u_\epsilon$ is the function defined in \eqref{eq:appellproof}.
\end{lemma}
\begin{proof}
  Formally, multiplying \eqref{eq:main3} by $\overline u$, integrating in $dx$ and taking the real part of the resulting identity, \eqref{eq:m} immediately follows.
  This argument is rigorous for solutions $u\in \mathcal C([0,1];H^1)$; a standard approximation argument permits to conclude the same for $L^2$-solutions.

  With the same argument, which is now rigorous since $\widetilde u_\epsilon$ is in $H^1$, by equation \eqref{eq:2appellproof} we easily obtain
  \begin{equation}\label{eq:2C1C2}
    \frac{d}{dt}\left\|\widetilde u_\epsilon
    (\cdot,t)\right\|_{L^2}
    \leq
    \epsilon\left(\left\|\widetilde V_\epsilon
    (\cdot,t)\right\|_{L^\infty}
    \left\|\widetilde u_\epsilon
    (\cdot,t)\right\|_{L^2}
    +\left\|\widetilde F_\epsilon
    (\cdot,t)\right\|_{L^2}
    \right),
  \end{equation}
  and by \eqref{eq:Vappellproof} we easily estimate
  \begin{equation}\label{eq:3C1C2}
    M_{1,\epsilon}:=
    \left\|\widetilde V_\epsilon
    (\cdot,t)\right\|_{L^\infty}
    \leq
    \sup_{t\in[0,1]}
    \frac{\alpha_\epsilon\beta_\epsilon}
    {(\alpha_\epsilon t+\beta_\epsilon(1-t))^2}
    \left\|V_1\right\|_{L^\infty}
    \leq
    \frac{\beta}{\alpha}M_1<\infty,
  \end{equation}
  $M_1$ being the constant defined in \eqref{eq:assV3}.
  Taking $\gamma=0$ in \eqref{eq:corappell1proof}, since $\alpha<\beta$ we get
  \begin{equation*}
    \left\|\widetilde u_\epsilon
    (\cdot,t)\right\|_{L^2}
    \leq
    \left\|u_\epsilon
    (\cdot,s)\right\|_{L^2},
  \end{equation*}
  and by the last inequality, together with \eqref{eq:te} and \eqref{eq:m} we conclude that
  \begin{equation}\label{eq:4C1C2}
    \left\|\widetilde u_\epsilon
    (\cdot,t)\right\|_{L^2}
    \leq
    e^{\epsilon\|V_1\|_{L^\infty}}
    \left\|u(\cdot,s)\right\|_{L^2}
    \leq
    e^{\epsilon\|V_1\|_{L^\infty}}
    N_1
    \left\|u(\cdot,0)\right\|_{L^2}.
  \end{equation}
  Arguing in a similar way, by \eqref{eq:corappell2proof} with $\gamma=0$, \eqref{eq:tf} and \eqref{eq:m} we get
  \begin{align}\label{eq:5C1C2}
    \left\|\widetilde F_\epsilon(\cdot,t)\right\|_{L^2}
    \leq
    \frac{\beta}{\alpha}
    \left\|F_\epsilon(\cdot,s)\right\|_{L^2}
    &
    \leq
    \frac{\beta}{\alpha}
    e^{\epsilon\|V_1\|_{L^\infty}}
    \left\|V_2(\cdot,s)\right\|_{L^\infty}
    \left\|u(\cdot,s)\right\|_{L^2}
    \\
    \nonumber
    &
    \leq
    \frac{\beta}{\alpha}
    e^{\epsilon\|V_1\|_{L^\infty}}
    \sup_{[t\in0,1]}
    \left\|V_2(\cdot,s)\right\|_{L^\infty}
    N_1\left\|u(\cdot,0)\right\|_{L^2}.
  \end{align}
  Inequality \eqref{eq:m2} now follows from \eqref{eq:2C1C2}, \eqref{eq:3C1C2}, \eqref{eq:4C1C2} and \eqref{eq:5C1C2}.
\end{proof}
We are finally ready to check the applicability of Lemma \ref{lem:3} to $\widetilde u_\epsilon$.

First, taking $\gamma=\frac{1}{\alpha_\epsilon\beta_\epsilon}=:
\gamma_{\epsilon}$ in \eqref{eq:corappell1proof}, since $\alpha<\beta$ we get
\begin{equation}\label{eq:4thm3}
  \left\|e^{\gamma_\epsilon|\cdot|^2}
  \widetilde u_\epsilon(\cdot,0)\right\|_{L^2}
  \leq
  \left\|e^{\frac{|\cdot|^2}{\beta_\epsilon^2}}
  u_\epsilon(\cdot,0)\right\|_{L^2}<\infty,
\end{equation}
by \eqref{eq:be} and \eqref{eq:decay3} (here we also used the fact that $s=0$ when $t=0$).

Analogously, by \eqref{eq:ae}, \eqref{eq:decay3} and the fact that $s=1$ when $t=1$, we obtain
\begin{equation}\label{eq:5thm3}
  \left\|e^{\gamma_\epsilon|\cdot|^2}
  \widetilde u_\epsilon(\cdot,1)\right\|_{L^2}
  \leq
  \left\|e^{\frac{|\cdot|^2}{\alpha_\epsilon^2}}
  u_\epsilon(\cdot,1)\right\|_{L^2}<\infty.
\end{equation}
Taking now $\gamma=\gamma_\epsilon$ in \eqref{eq:corappell2proof}, by \eqref{eq:f} and \eqref{eq:m} we easily estimate
\begin{equation}\label{eq:6thm3}
  \left\|e^{\gamma_\epsilon|\cdot|^2}
  \widetilde F_\epsilon(\cdot,t)\right\|_{L^2}
  \leq
  \frac{\beta}{\alpha}
  e^{\epsilon\|V_1\|_{L^\infty}}
  \left\|e^{\frac{|\cdot|^2}{
  (\alpha s+\beta(1-s))^2}}V_2(\cdot,s)\right\|_{L^\infty}
  N_1\|u(\cdot,0)\|_{L^2}.
\end{equation}
On the other hand, taking $\gamma=0$ in \eqref{eq:corappell1proof} gives
\begin{equation}\label{eq:7thm3}
  \lim_{\epsilon\to0}\left\|
  \widetilde u_\epsilon(\cdot,t)\right\|_{L^2}
  =
  \lim_{\epsilon\to0}\left\|
  e^{\frac{(\alpha_\epsilon-\beta_\epsilon)\epsilon}
  {4(\epsilon^2+1)(\alpha_\epsilon s+\beta_\epsilon
  (1-s))}|\cdot|^2}
   u_\epsilon(\cdot,s)\right\|_{L^2}
   =
   \|u(\cdot,s)\|_{L^2}.
\end{equation}
Now, by Lemma \ref{lem:C1C2}
\begin{equation}\label{eq:8thm3}
  \frac{d}{dt}\left\|\widetilde u_\epsilon(\cdot,t)\right\|
  _{L^2}\leq C,
\end{equation}
for some $C=C\left(\epsilon,\alpha,\beta,\|V_1\|_{L^\infty},
\|u(\cdot,0)\|_{L^2},
\sup_{t\in[0,1]}\|V_2(\cdot,t)\|_{L^\infty}\right)$ and for any $t\in[0,1]$.
By \eqref{eq:7thm3} and \eqref{eq:8thm3} we hence obtain that
\begin{equation*}
  \left\|
  \widetilde u_\epsilon(\cdot,t)\right\|_{L^2}
  \to
  \left\|
   u(\cdot,s)\right\|_{L^2},
\end{equation*}
as $\epsilon\to0$, uniformly in $[0,1]$, and in particular, by \eqref{eq:m}, there exists $0<\epsilon_0=\epsilon_0(\|u(\cdot,0)\|_{L^2},N_1)$ such that
\begin{equation}\label{eq:9thm3}
  \left\|
  \widetilde u_\epsilon(\cdot,t)\right\|_{L^2}
  \geq
  \frac{\|u(\cdot,0)\|_{L^2}}{2N_1},
\end{equation}
for any $t\in[0,1]$ and any $\epsilon\in(0,\epsilon_0)$.
By \eqref{eq:6thm3} and \eqref{eq:9thm3} we finally obtain
\begin{equation}\label{eq:10thm3}
  M_2{,\epsilon}
  :=
  \sup_{t\in[0,1]}\frac{\left\|e^{\gamma_\epsilon|\cdot|^2}
  \widetilde F_\epsilon(\cdot,t)\right\|_{L^2}}
  {\left\|
  \widetilde u_\epsilon(\cdot,t)\right\|_{L^2}}
  \leq
  2N_1^2
  \frac{\beta}{\alpha}
  e^{\epsilon\|V_1\|_{L^\infty}}
  \sup_{s\in[0,1]}\left\|e^{\frac{|\cdot|^2}{
  (\alpha s+\beta(1-s))^2}}V_2(\cdot,s)\right\|_{L^\infty}
  <\infty,
\end{equation}
by assumptions \eqref{eq:assV3}, \eqref{eq:assV23}, for $\epsilon>0$ small enough; this ensures the validity of the second condition \eqref{eq:3V} (the first condition in \eqref{eq:3V} is quite immediate, thanks to \eqref{eq:Vappellproof} and \eqref{eq:assV3}).

We now pass to condition \eqref{eq:3A}. By \eqref{eq:Aappellproof}, writing $g_\epsilon(t):=\sqrt
{\alpha_\epsilon\beta_\epsilon}/(\alpha_\epsilon(1-t)
+\beta_\epsilon t)$ we explicitly compute
\begin{align}
\label{eq:nuova1}
  \partial_t\widetilde A_\epsilon(x,t)
  &
  =
  g_\epsilon'(t)\left[
  A(xg_\epsilon(t))+g_\epsilon(t)
  x^tDA(xg_\epsilon(t))\right]
  \\
  \label{eq:nuova2}
  x^t\widetilde B_\epsilon(x,t)
  &
  :=
  x^t(D\widetilde A_\epsilon-D\widetilde A_\epsilon^t)(x,t)
  =
  g_\epsilon^2(t)x^tB(xg_\epsilon(t)).
\end{align}
Writing
\begin{equation*}
  g_\epsilon'(t)=
  \frac{(\alpha_\epsilon-\beta_\epsilon)}{\sqrt
  {\alpha_\epsilon\beta_\epsilon}}g_\epsilon^2(t)
\end{equation*}
and estimating
\begin{equation*}
  \sup_{t\in[0,1]}g_\epsilon^2(t)\leq\frac{\beta_\epsilon}
  {\alpha_\epsilon},
\end{equation*}
we easily obtain, using the above identities, \eqref{eq:assA3}
and \eqref{eq:assA3new},
\begin{align}
  \label{eq:15thm3}
  M_{\widetilde A,\epsilon}
  &
  :=
  \frac{1}{\gamma_\epsilon}
  \sup_{t\in[0,1]}\left\|\partial_t\widetilde A_\epsilon
  (\cdot,t)\right\|_{L^\infty}^2
  +4\gamma_\epsilon(\epsilon^2+1)
  \sup_{t\in[0,1]}\left\|x^t\cdot\widetilde B(\cdot,t)\right\|_{L^\infty}^2
  \\
  \nonumber
  &
  \leq
  \frac{2(\alpha_\epsilon^2
  +\beta_\epsilon^2)\beta_\epsilon^2}{\alpha_\epsilon^2}
  \left(\|A\|_{L^\infty}^2+\|x^tDA\|_{L^\infty}^2\right)
  +\frac{4(\epsilon^2+1)}{\alpha_\epsilon^2}\|x^tB\|
  _{L^\infty}^2
  \\
  \nonumber
  &
  \leq
  \frac{4}{\alpha_\epsilon^2}
  \left[(\alpha_\epsilon^2
  +\beta_\epsilon^2)\beta_\epsilon^2+\epsilon^2+1\right]
  M_A<\infty,
\end{align}
$M_A$ being the constant in \eqref{eq:assA3}.

Finally, notice that from (\ref{eq:gauge3}) and
(\ref{eq:Aappellproof}), and from (\ref{eq:gauge32}) and
(\ref{eq:nuova1}) it follows that
$$
  x\cdot \widetilde A_{\epsilon}\equiv 0,
  \qquad {\hbox {and}} \qquad
  x\cdot \partial_t \widetilde A_{\epsilon}\equiv 0,
$$
respectively.

The  above argument shows that we can apply the results in Lemmata
\ref{lem:3} and \ref{lem:4} to obtain
%

\begin{align}
  \label{eq:17thm3}
  &
  \left\|e^{\gamma_\epsilon|\cdot|^2}
  \widetilde u_\epsilon(\cdot,t)\right\|_{L^2}
  \\
  \nonumber
  & \ \
  \leq
  e^{N_1\left[
  M_{\widetilde A,\epsilon}+
  \sqrt{\epsilon^2+1}(M_{1,\epsilon}+M_{2,\epsilon})
  +(\epsilon^2+1)\left(M_{1,\epsilon}^2+M_{2,\epsilon}^2\right)\right]}
  \left\|e^{\gamma_\epsilon|\cdot|^2}
  \widetilde u_\epsilon(\cdot,0)\right\|_{L^2}^{1-t}
  \left\|e^{\gamma_\epsilon|\cdot|^2}
  \widetilde u_\epsilon(\cdot,1)\right\|_{L^2}^t
  \\
  \label{eq:18thm3}
  &
  \left\|\sqrt{t(1-t)}e^{\gamma_\epsilon|x|^2}
  \nabla_{\widetilde A_\epsilon}
  \widetilde u_\epsilon(x,t)\right\|_{L^2(\R^n\times[0,1])}
  +\gamma_\epsilon
  \left\|\sqrt{t(1-t)}e^{\gamma_\epsilon|x|^2}
  |x|
  \widetilde u_\epsilon(x,t)\right\|_{L^2(\R^n\times[0,1])}
  \\
  \nonumber
  & \ \
  \leq N_{2,\epsilon}
  \left[(M_{1,\epsilon}+1)\sup_{t\in[0,1]}
  \left\|e^{\gamma_\epsilon|\cdot|^2}
  \widetilde u_\epsilon(\cdot,t)\right\|_{L^2}
  +\sup_{t\in[0,1]}
  \left\|e^{\gamma_\epsilon|\cdot|^2}
  \widetilde F_\epsilon(\cdot,t)\right\|_{L^2}
  \right],
\end{align}
with $N_1$ an universal constant and $N_{2,\epsilon}=N_{2,\epsilon}
\left(\epsilon,\gamma_\epsilon\right)>0$.

\subsection{Step IV: conclusion of the proof}\label{subsec:conclusion}

It is now simple to conclude the proof of Theorem \ref{thm:3}. Indeed, it is sufficient to rewrite estimates \eqref{eq:17thm3} and \eqref{eq:18thm3} in terms of the function $u_\epsilon(t)$, using Corollary \ref{cor:appell}; finally, \eqref{eq:31} and \eqref{eq:32} follow by taking the limit as $\epsilon$ tends to $0$. We omit further details.

\begin{remark}\label{rem:alfabeta}
  In the case $\alpha=\beta$ the same proof as above holds, in a much simpler version. Indeed, in this case
  it is useless to apply the Appell transformation and the proof can be directly performed on the
  function $u_\epsilon$, by means of Lemmata \ref{lem:3} and \ref{lem:4}.
\end{remark}

\section{Proof of Theorem \ref{thm:hardy}}\label{sec:proof}

\begin{lemma}[Carleman estimate]\label{thm:carleman}
  Let $n\geq3$, $A=A(x,t):\R^{n+1}\to\R^n$, denote by $B=DA-DA^t$ and assume that $x^tB\in L^\infty$. In addition, assume that 
  \begin{equation}\label{eq:ortogcarleman}
  x\cdot A_t(x)\equiv0,
  \qquad
  v\cdot A_t(x)\equiv0,
  \qquad
   \qquad {\hbox{and}}\qquad
     v^tB(x)\equiv0,
  \end{equation}
  for any $x\in\R^n$ and some unit vector $v\in\mathcal S^{n-1}$.
  Then, for any $\epsilon>0$, $\mu>0$, $g=g(x,t)\in\mathcal C^\infty_0(\R^{n+1})$, and $R>8\mu\epsilon^{-\frac12}\left\|x^tB\right\|_{L^\infty}$, the following inequality holds:
  \begin{align}
    \label{eq:carleman}
    &
    \frac{R}{4}\sqrt{\frac{\epsilon}{\mu}}
    \left\|e^{\mu\left|x+Rt(1-t)v\right|^2
    -\frac{(1+\epsilon)R^2t(1-t)}{16\mu}}g(x,t)\right\|
    _{L^2\left(\R^{n+1}\right)}
    \\
    \nonumber
    & \ \
    \leq
    \left\|e^{\mu\left|x+Rt(1-t)v\right|^2
    -\frac{(1+\epsilon)R^2t(1-t)}{16\mu}}
    \left(\partial_t-i\Delta_A\right)
    g(x,t)\right\|
    _{L^2\left(\R^{n+1}\right)}.
  \end{align}
\end{lemma}
\begin{proof}
  For simplicity, we can assume without loss of generality that $v=e_1=(1,0,\dots,0)$. Let
  \begin{equation*}
    f(x,t):=e^{\mu\left|x+Rt(1-t)e_1\right|^2
    -\frac{(1+\epsilon)R^2t(1-t)}{16\mu}}g(x,t).
  \end{equation*}
  Then we have
  \begin{equation}\label{eq:serve}
    e^{\mu\left|x+Rt(1-t)e_1\right|^2
    -\frac{(1+\epsilon)R^2t(1-t)}{16\mu}}
    \left(\partial_t-i\Delta_A\right)
    g
    =
    \left(\partial_t-\mathcal S-\mathcal A\right)f,
  \end{equation}
  $\mathcal S$ and $\mathcal A$ being the ones in \eqref{eq:S} and \eqref{eq:A}, respectively, with $a=0$ and $b=1$.

  Following the usual method to prove Carleman estimates (see \cite{H}), we now write
  \begin{align}
    \label{eq:conjugation}
     &
     \left\|\left(\partial_t-\mathcal S-\mathcal A\right)f
     \right\|_{L^2(\R^{n+1})}^2
     \\
     \nonumber
     & \ \
     =
     \left\|\left(\partial_t-\mathcal A\right)f
     \right\|_{L^2(\R^{n+1})}^2
     +
     \left\|\mathcal Sf
     \right\|_{L^2(\R^{n+1})}^2
     -2\Re\int\int
     \mathcal Sf\overline{\left(\partial_t-
     \mathcal A\right)f}\,dx\,dt
     \\
     \nonumber
     & \ \
     \geq
     \int\int\left(\mathcal S_t+[\mathcal S,\mathcal A]\right)f\,\overline f\,dx\,dt.
  \end{align}
Applying now \eqref{eq:St} and \eqref{eq:commutator} with the choices $a=0$, $b=1$, and
\begin{equation*}
  \varphi(x,t)=\mu\left|x+Rt(1-t)e_1\right|^2
    -\frac{(1+\epsilon)R^2t(1-t)}{16\mu},
\end{equation*}
noticing that $\nabla\varphi\cdot A_t\equiv0$ by the first
two conditions in \eqref{eq:ortogcarleman}, an easy computation
involving the completion of two squares leads to
\begin{align}
  \label{eq:quasifin}
  &
  \int\int\left(\mathcal S_t+[\mathcal S,\mathcal A]\right)f\,\overline f\,dx\,dt
  \\
  \nonumber
  &
  =
  32\mu^3\int\int|f|^2\left|
  x+Rt(1-t)e_1-\frac{R}{16\mu^2}e_1\right|^2
  +\frac{\epsilon R^2}{8\mu}\int\int|f|^2
  +8\mu\int\int\left|\nabla_{A,x'}f\right|^2
  \\
  \nonumber
  & \ \
  +8\mu\int\int\left|\partial^1_Af+i\frac{R(1-2t)}
  {2}f\right|^2
  +8\mu\Im\int\int f\left(x+Rt(1-t)e_1\right)^tB\cdot
  \overline{\nabla_Af},
\end{align}
where $\nabla_A=\nabla-iA:=(\partial_A^1,\dots,\partial_A^n)$, $\nabla_{A,x'}:=(0,\partial_A^2,\dots,\partial_A^n)$.
Notice that, since $e_1^tB=0$ and $B$ is anti-symmetric, we can
write
\begin{align}\label{eq:quasifin2}
  f\left(x+Rt(1-t)e_1\right)^tB\cdot
  \overline{\nabla_Af}
  &
  =
  fx^tB\cdot
  \overline{\nabla_Af}
  =
   fx^tB\cdot\left(\overline{\nabla_Af}
  +i\frac{R(1-2t)}{2}e_1f\right)
  \\
  &
  =
  fx^tB\cdot\overline{\nabla_{A,x'}f}
  +fx^tB\cdot\left(\partial^1_Af+
  i\frac{R(1-2t)}{2}f\right) e_1.
\end{align}
Therefore, by Cauchy-Schwartz and the elementary inequality
$ab\leq\delta a^2+\frac1{4\delta} b^2$, with the choice $\delta :=
8\mu$, we can estimate
\begin{align}
  \label{eq:quasifin3}
  &
  \left|
  8\mu\Im\int\int f\left(x+Rt(1-t)e_1\right)^tB\cdot
  \overline{\nabla_Af}
  \right|
  \\
  \nonumber
  & \ \
  \leq
  4\mu\|x^tB\|_{L^\infty}^2\int\int|f|^2
  +8\mu\int\int\left|\partial^1_Af+i\frac{R(1-2t)}
  {2}f\right|^2
  +8\mu\int\int\left|\nabla_{A,x'}f\right|^2.
\end{align}
In conclusion, by \eqref{eq:quasifin} and \eqref{eq:quasifin3}, neglecting the term with cubic growth in $\mu$ we get
\begin{equation*}
  \int\int\left(\mathcal S_t+[\mathcal S,\mathcal A]\right)f\,\overline f
  \geq
  \left[\frac{\epsilon R^2}{8\mu}
  -4\mu\|x^tB\|_{L^\infty}^2\right]\int\int|f|^2.
\end{equation*}
The last inequality, together with \eqref{eq:serve}, \eqref{eq:conjugation} and the condition $R>8\mu\epsilon^{-\frac12}\|x^tB\|_{L^\infty}$, completes the proof of \eqref{eq:carleman}.
\end{proof}
\begin{proof}[Proof of Theorem \ref{thm:hardy}]
  With the tools introduced up to now, the proof of Theorem \ref{thm:hardy} is now reduced to a typical argument in the Carleman's spirit.

  Let $u\in\mathcal C([0,1];L^2(\R^n))$ be the solution to \eqref{eq:main3hardy}. As in the first step of the previous section, we first reduce to the Cronstr\"om gauge, passing from $A$ to $\widetilde A$ by means of Lemma \ref{lem:cronstrom1}. It is hence sufficient to prove that $\tilde u=e^{i\varphi}u\equiv0$, where $\varphi$ is given by \eqref{eq:varphi}. From now on, by abuse of notation, we keep calling $u$ the gauged function $\widetilde u$ and by $A$ the transformed potential $\widetilde A$, which satisfy identities \eqref{eq:cronstrom2}, \eqref{eq:cronstrom3}, \eqref{eq:gaugefinal}.

Now apply the Appell transformation (Lemma \ref{lem:appell}) with $a=0$ and $b=1$, to obtain the new
function $\widetilde u$ in \eqref{eq:appell}, satisfying
\begin{equation*}
  \partial_t\widetilde u = i\left(\Delta_{\widetilde A}u+\widetilde V\tilde u\right),
\end{equation*}
where $\widetilde A$ and $\widetilde V$ are defined by
\eqref{eq:Aappell} and \eqref{eq:Vappell}, respectively, and
$V:=V_1+V_2$. Assumption \eqref{eq:decay3hardy} then gives that
$\|e^{\gamma|x|^2}\tilde u(0)\|_{L^2}+
  \|e^{\gamma|x|^2}\tilde u(1)\|_{L^2}<\infty$, for any $\gamma>\frac12$.

  In addition, by estimates \eqref{eq:17thm3} and \eqref{eq:18thm3}, in the limit as $\epsilon$ tends to 0, we have
  \begin{equation}\label{eq:thm3}
    \sup_{t\in[0,1]}\left\|e^{\gamma|\cdot|^2}
    \widetilde u(\cdot,t)
    \right\|_{L^2}
    +
    \left\|\sqrt{t(1-t)}
    e^{\gamma|\cdot|^2}\nabla_{\widetilde A }\widetilde u(\cdot,t)
    \right\|_{L^2(\R^n\times[0,1])}=:N_\gamma<\infty.
  \end{equation}
 Now, let $R>8\mu\epsilon^{-\frac12}\|x^tB\|_{L^\infty}$, as in the statement of Lemma \ref{thm:carleman}, and
let $M >0$, to be chosen later. Then, localize the function
$\widetilde u$ as follows: let $\theta_M(x),\eta_R(t)$ be two smooth
functions such that
  \begin{equation*}
    \theta_M\equiv1\ \text{if }|x|\leq M
    \qquad
    \theta_M\equiv0\ \text{if }|x|\geq 2M
  \end{equation*}
  \begin{equation*}
    \eta_R(t)\equiv 1 \ \text{if }t\in\left[
    \frac1R,1-\frac1R\right]
    \qquad
    \eta_R(t)\equiv 0 \ \text{if }t\in\left[
    0,\frac1{2R}\right]\cup\left[1-\frac1{2R},1\right],
  \end{equation*}
  and define
  \begin{equation*}
    g(x,t)=\theta_M(x)\eta_R(t)\widetilde u(x,t).
  \end{equation*}
  It turns out that $g$ solves
  \begin{equation}\label{eq:equationggg}
    \left(\partial_t-i\Delta_{\widetilde A}\right)g
    =
    i\widetilde Vg
    +\theta_M\eta_R'\widetilde u
    -i\left(2\nabla\theta_M\cdot\nabla_{\widetilde A}\widetilde u
    +\widetilde u\Delta\theta_M\right)\eta_R.
  \end{equation}
Assume without loss of generality that the magnetic field $B$
satisfies the condition \eqref{eq:ortoghardy} with $v=e_1$.

  Now choose
  \begin{equation}\label{eq:primamu}
  \mu\leq\frac{\gamma}{1+\epsilon},
  \end{equation}
  for some fixed small $\epsilon>0$.
Notice that, in the support of the second term of the right-hand
side of \eqref{eq:equationggg}, we have
\begin{equation}\label{eq:blabla}
  \mu\left|x+Rt(1-t)e_1\right|^2-\frac{(1+\epsilon)R^2t(1-t)}{16\mu} \leq \gamma|x|^2+\frac{\gamma}{\epsilon};
\end{equation}
analogously, in the support of the last term of the right-hand side of \eqref{eq:equationggg} we have
\begin{equation}\label{eq:blablabla}
  \mu\left|x+Rt(1-t)e_1\right|^2-\frac{(1+\epsilon)R^2t(1-t)}{16\mu} \leq \gamma|x|^2+\frac{\gamma R^2}{\epsilon}.
\end{equation}
By condition \eqref{eq:ortoghardy} with $v=e_1$,
  \eqref{eq:cronstrom3}, \eqref{eq:gaugefinal}, identity \eqref{eq:nuova1} with $\epsilon=0$ and the fact that $B$ is anti-symmetric, we get $x\cdot\partial_t\widetilde A\equiv0\equiv e_1\cdot \partial_t\widetilde A$.
  Hence, applying \eqref{eq:carleman} to $g$, by \eqref{eq:equationggg}, \eqref{eq:blabla}, \eqref{eq:blablabla} and the bounds for $\theta_M,\eta_R$ and their derivatives we easily get
  \begin{align}
  \label{eq:quasilast}
    &
    R\left\|e^{\mu\left|x+Rt(1-t)v\right|^2
    -\frac{(1+\epsilon)R^2t(1-t)}{16\mu}}g\right\|
    _{L^2\left(\R^n\times[0,1]\right)}
    \\
    \nonumber
    &
    \leq
    N_{\epsilon,\mu}\left\|\widetilde V\right\|_{L^\infty(\R^n\times[0,1])}
    \left\|e^{\mu\left|x+Rt(1-t)v\right|^2
    -\frac{(1+\epsilon)R^2t(1-t)}{16\mu}}g\right\|
    _{L^2\left(\R^n\times[0,1]\right)}
    \\
    \nonumber
    &
    +N_{\epsilon,\mu} Re^{\frac\gamma\epsilon}
    \sup_{[0,1]}\left\|
    e^{\gamma|\cdot|^2}\widetilde u(\cdot,t)\right\|
    _{L^2}
    +N_\epsilon M^{-1}
    e^{\frac{\gamma R^2}{\epsilon}}
    \left\|
    e^{\gamma|x|^2}\left(|\widetilde u|
    +|\nabla_{\widetilde A}\widetilde u|\right)\right\|
    _{L^2\left(\R^n\times\left[\frac{1}{2R},
    1-\frac{1}{2R}\right]\right)},
  \end{align}
  with $N_{\epsilon,\mu}=4\sqrt{\mu/\epsilon}$.
  Notice that, choosing $R\geq2N_\epsilon\left\|\widetilde V\right\|_{L^\infty(\R^n\times[0,1])}$, the first term in the right-hand side of the last inequality can be hidden in the left-hand side. Moreover, by \eqref{eq:thm3}, we have that
  \begin{equation}\label{eq:tails}
    \lim_{M\to\infty}
    N_\epsilon M^{-1}
    e^{\frac{\gamma R^2}{\epsilon}}
    \left\|
    e^{\gamma|x|^2}\left(|\widetilde u|
    +|\nabla_{\widetilde A}\widetilde u|\right)\right\|
    _{L^2\left(\R^n\times\left[\frac{1}{2R},
    1-\frac{1}{2R}\right]\right)}
    =0,
  \end{equation}
  for any fixed $R$. Finally, choose
  \begin{equation*}
    M:=f(\epsilon)\frac R8,
  \end{equation*}
  for some positive function $f(\epsilon)$ such that $f(\epsilon)<1-\epsilon^2$ and $f(\epsilon)\to0$ as $\epsilon$ tends to 0.
  Notice that $g\equiv\widetilde u$ in $B_{f(\epsilon)R/8\times[(1-\epsilon)/2,(1+\epsilon)/2]}$; in this set, one can easily estimate
   \begin{align*}
    &
    \mu\left|x+Rt(1-t)v\right|^2
    -\frac{(1+\epsilon)R^2t(1-t)}{16\mu}
    \\
    & \ \ \ 
    \geq
    \frac{R^2}{16\mu}
    \left\{\mu^2\left[(1-\epsilon^2)^2-(1-\epsilon^2)f(\epsilon)\right]-\frac14(1+\epsilon)\right\}.
  \end{align*}Consequently, choosing
\begin{equation}\label{eq:secondamu}
  \mu^2>\frac14\cdot\frac{1+\epsilon}{(1-\epsilon^2)^2-(1-\epsilon^2)f(\epsilon)}
\end{equation}
one obtains that
 \begin{equation*}
    \mu\left|x+Rt(1-t)v\right|^2
    -\frac{(1+\epsilon)R^2t(1-t)}{16\mu}
    \geq
    0
  \end{equation*}
  in $B_{f(\epsilon)R/8\times[(1-\epsilon)/2,(1+\epsilon)/2]}$, in which we also have $g\equiv\widetilde u$.
  Comparing \eqref{eq:primamu} and \eqref{eq:secondamu}, we see that they are compatible if and only if $\gamma>\frac12$, i.e. $\alpha\beta>2$, as required in the statement of Theorem \ref{thm:hardy}.

  Therefore, by \eqref{eq:quasilast}, and the above considerations, there exist $C(\gamma,\epsilon), N_{\gamma,\epsilon}>0$ such that
  \begin{equation}\label{eq:penultima}
    Re^{C(\gamma,\epsilon) R^2}\left\|\widetilde
    u(x,t)\right\|_{L^2\left(B_{\frac R8}\times
    \left[\frac{1-\epsilon}{2},\frac{1+\epsilon}{2}
    \right]\right)}
    \leq
    N_{\gamma,\epsilon}R,
  \end{equation}
  for any $R>\max\{8\mu\epsilon^{-\frac12}\|x^tB\|_{L^\infty},
  \ 2N_\epsilon\|\widetilde V\|_{L^\infty(\R^n\times[0,1])}\}$.
  By \eqref{eq:m} in Lemma \ref{lem:C1C2}, \eqref{eq:thm3} and \eqref{eq:penultima} we now conclude that there exists a constant $N=N(\gamma,\epsilon,V)$ depending on $N_\gamma$, $\epsilon$ and $\sup_{[0,1]}\|V\|_{L^\infty}$ such that
  \begin{equation*}
    e^{C(\gamma,\epsilon)R^2}\left\|\widetilde
    u(\cdot,0)\right\|_{L^2}
    \leq
    N(\gamma,\epsilon,V).
  \end{equation*}
  Letting $R$ tend to infinity, this implies that $\widetilde u\equiv u\equiv0$.
\end{proof}

\end{document}